\documentclass[11pt,oneside,english,reqno,twoside]{amsart}
\usepackage[T1]{fontenc}
\usepackage[latin9]{luainputenc}
\usepackage{geometry}
\usepackage{amstext}
\usepackage{amsthm}
\usepackage{amssymb}

\makeatletter
\numberwithin{equation}{section}
\numberwithin{figure}{section}
\theoremstyle{plain}
\newtheorem{thm}{\protect\theoremname}[section]
  \theoremstyle{definition}
  \newtheorem{defn}[thm]{\protect\definitionname}
  \theoremstyle{plain}
  \newtheorem{lem}[thm]{\protect\lemmaname}
  \theoremstyle{remark}
  \newtheorem*{rem*}{\protect\remarkname}
  \theoremstyle{plain}
  \newtheorem{prop}[thm]{\protect\propositionname}
  \theoremstyle{remark}
  \newtheorem{rem}[thm]{\protect\remarkname}
  \theoremstyle{plain}
  \newtheorem{cor}[thm]{\protect\corollaryname}

\usepackage{hyperref}

\usepackage{titleps}

\sethead[\thepage][\Author][]{}{\Title}{\thepage}

\newcommand{\Addresses}{{
  \bigskip
  \footnotesize
  \textsc{Jarkko Siltakoski,  Department of Mathematics and Statistics, P.O.Box 35, FIN-40014, University of Jyväskylä, Finland}\par\nopagebreak
  \textit{E-mail address}: \href{mailto:jarkko.j.m.siltakoski@student.jyu.fi}{jarkko.j.m.siltakoski@student.jyu.fi}
}}

\makeatother

\usepackage{babel}
  \providecommand{\corollaryname}{Corollary}
  \providecommand{\definitionname}{Definition}
  \providecommand{\lemmaname}{Lemma}
  \providecommand{\propositionname}{Proposition}
  \providecommand{\remarkname}{Remark}
\providecommand{\theoremname}{Theorem}

\begin{document}
\global\long\def\d{\,d}
\global\long\def\tr{\mathrm{tr}}
\global\long\def\supp{\mathrm{supp}\thinspace}
\global\long\def\div{\text{div}}
\renewcommand{\labelenumi}{(\roman{enumi})} 

\title[Equivalence of viscosity and weak solutions]{Equivalence of viscosity and weak solutions for the normalized $p(x)$-Laplacian}

\author{Jarkko Siltakoski}
\begin{abstract}
We show that viscosity solutions to the normalized $p(x)$-Laplace
equation coincide with distributional weak solutions to the strong
$p(x)$-Laplace equation when $p$ is Lipschitz and $\inf p>1$. This
yields $C^{1,\alpha}$ regularity for the viscosity solutions of the
normalized $p(x)$-Laplace equation. As an additional application,
we prove a Radó-type removability theorem.
\end{abstract}

\maketitle

\section{Introduction}

In this paper, we study viscosity solutions to the \textit{normalized
$p(x)$-Laplace equation} which is defined by
\begin{equation}
-\Delta_{p(x)}^{N}u:=-\Delta u-\frac{p(x)-2}{\left|Du\right|^{2}}\Delta_{\infty}u:=-\Delta u-\frac{p(x)-2}{\left|Du\right|^{2}}\left\langle D^{2}uDu,Du\right\rangle =0.\label{eq:normalizedpx}
\end{equation}
There has been recent interest in normalized equations, see for example
\cite{jinsilvestre17,imbertJinSilvestre16,banerjeeGarofalo15}. We
are partly motivated by the connection to stochastic tug-of-war games
\cite{peresSchrammSheffieldWilson09} as the case of space dependent
probabilities leads to (\ref{eq:normalizedpx}) \cite{gamestuff}.

The objective of this work is to show that viscosity solutions to
(\ref{eq:normalizedpx}) coincide with solutions of its counterpart
in the theory of distributional weak solutions. One approach to this
kind of equivalence results \cite{equivalence_plaplace,Ishii95} is
based on the uniqueness of solutions. However, it seems difficult
to use uniqueness in our case because the uniqueness of solutions
is an open problem for the equation (\ref{eq:normalizedpx}) as pointed
out in \cite{equivalence_p(x)Laplace}. The equation (\ref{eq:normalizedpx})
is in the non-divergence form. In order to find the weak counterpart,
we note that for $u\in C^{2}(\Omega)$ with non-vanishing gradient
it holds that
\begin{align*}
-\left|Du\right|^{p(x)-2}\Delta_{p(x)}^{N}u= & \mathrm{-div}\left(\left|Du\right|^{p(x)-2}Du\right)+\left|Du\right|^{p(x)-2}\log\left(\left|Du\right|\right)Du\cdot Dp.
\end{align*}
Thus the weak counterpart of (\ref{eq:normalizedpx}) should be the
\textit{strong $p(x)$-Laplace equation}
\begin{equation}
-\Delta_{p(x)}^{S}u:=-\mathrm{div}(\left|Du\right|^{p(x)-2}Du)+\left|Du\right|^{p(x)-2}\log\left|Du\right|Du\cdot Dp=0.\label{eq:strongpx}
\end{equation}

Our main result, Theorem \ref{thm:equivalence}, is that viscosity
solutions to (\ref{eq:normalizedpx}) coincide with weak solutions
to (\ref{eq:strongpx}) when the function $p$ is Lipschitz with $\inf p>1$.
With these assumptions weak solutions to (\ref{eq:strongpx}) in a
domain are locally $C^{1,\alpha}$ continuous \cite{strongpx_regularity}.
Thus our equivalence result yields local $C^{1,\alpha}$ regularity
also for viscosity solutions to (\ref{eq:normalizedpx}). As an application,
we prove a Radó-type removability theorem for the strong $p(x)$-Laplacian.
The theorem follows from the equivalence result since in the definition
of a viscosity solution we may ignore the test functions whose gradient
vanishes.

That viscosity solutions to (\ref{eq:normalizedpx}) are weak solutions
to (\ref{eq:strongpx}) is proven by applying the method of \cite{newequivalence}.
The idea is to approximate a viscosity solution through a sequence
of inf-convolutions, show that the inf-convolutions are essentially
weak supersolutions, and then pass to the limit.

First, in Lemma \ref{lem:ishiiq} we show that the inf-convolution
$u_{\varepsilon}$ of a viscosity supersolution $u$ to (\ref{eq:normalizedpx})
is still, in essence, a viscosity supersolution up to some error.
This fact is a key part of our proof. If there was no $x$-dependence
in (\ref{eq:normalizedpx}), it would be straightforward to see that
the inf-convolution of a viscosity supersolution is still a viscosity
supersolution. This is because a test function that touches the inf-convolution
from below also touches the original function from below at a nearby
point once we add some constant to it. From this it would follow that
the inf-convolution is a supersolution to the original equation. However,
the equation (\ref{eq:normalizedpx}) has $x$-dependence caused by
$p(x)$. Thus the inf-convolution no longer satisfies the original
equation.

In Lemma \ref{lem:proof_lemma} we use the standard mollification
on $u_{\varepsilon}$ and $p$ to deduce from Lemma \ref{lem:ishiiq}
that $u_{\varepsilon}$ is ``almost'' a weak solution to $-\Delta_{p(x)}^{S}u_{\varepsilon}\geq0$.
Applying Caccioppoli type estimates and vector inequalities we are
then able to deduce that the sequence of inf-convolutions converges
to the viscosity supersolution in $W_{loc}^{1,p(\cdot)}(\Omega)$
as $\varepsilon\rightarrow0$. This allows us to pass to the limit
and conclude that $u$ satisfies $-\Delta_{p(x)}^{S}u\geq0$ in the
weak sense.

Due to the variable exponent, the operator $\Delta_{p(x)}^{S}$ can
be singular in some subsets and degenerate in others. Therefore we
apply different arguments in the cases $p(x)<2$ and $p(x)\geq2$,
and finally need to be able to combine them.

The equivalence of weak and viscosity solutions to the usual $p$-Laplace
equation was first proven by Juutinen, Lindqvist and Manfredi \cite{equivalence_plaplace}.
Later Julin and Juutinen \cite{newequivalence} presented a more direct
way to show that viscosity solutions to $-\Delta_{p}u=f$ are also
weak solutions. This proof was adapted in \cite{OptimalC1} to show
that viscosity solutions to $-\Delta_{p}^{N}u=f$ coincide with weak
solutions to $-\Delta_{p}u=\left|Du\right|^{p-2}f$ when $p\geq2$.
Similar arguments were also used in \cite{chilepaper} to study the
equivalence of solutions to $-\Delta_{p}u=f(x,u,Du)$. The variable
exponent case was explored in \cite{equivalence_p(x)Laplace} where
the equivalence of weak and viscosity solutions was proven for the
$p(x)$-Laplace equation using techniques of \cite{equivalence_plaplace}.

The equation (\ref{eq:strongpx}) was introduced by Adamowicz and
Hästö \cite{intr_strongpx_1,intr_strongpx_2} in connection with mappings
of finite distortion. It has been further studied for example in \cite{strongpx_regularity,somepxpaper}. 

The paper is organized as follows: in Section 2 we recall the variable
exponent Lebesgue and Sobolev spaces. Section 3 contains the rigorous
definitions of solutions to equations (\ref{eq:normalizedpx}) and
(\ref{eq:strongpx}). In Section 4 we show that weak solutions of
(\ref{eq:normalizedpx}) are viscosity solutions to (\ref{eq:strongpx})
and the converse statement is proven in Section 5. Finally, in Section
6 we formulate and prove a Radó-type removability theorem for weak
solutions of (\ref{eq:strongpx}).

\section{Variable exponent lebesgue and sobolev spaces}

We briefly recall basic facts about these spaces. For general reference
see e.g. \cite{pxbook}. Let $\Omega\subset\mathbb{R}^{N}$ be an
open and bounded set and let $p:\Omega\rightarrow(1,\infty$) be a
measurable function. We denote

\[
p^{+}:=\underset{x\in\Omega}{\mathrm{ess\,sup}}\,p(x)\ \ \text{and}\ \ \underset{x\in\Omega}{p^{-}:=\mathrm{ess\,inf}\,p(x).}
\]
The \textit{variable exponent Lebesgue space} $L^{p(\cdot)}(\Omega)$
is defined as the set of measurable functions $u:\Omega\rightarrow\mathbb{R}$
for which the $p(\cdot)$-modular
\[
\varrho_{p(\cdot)}(u):=\int_{\Omega}\left|u\right|^{p(x)}\d x
\]
is finite. It is a Banach space equipped with the Luxemburg norm
\[
\left\Vert u\right\Vert _{L^{p(\cdot)}(\Omega)}:=\inf\left\{ \lambda>0:\int_{\Omega}\left|\frac{u}{\lambda}\right|^{p(x)}\d x\leq1\right\} .
\]
Given that $p^{+}<\infty$ or $\varrho_{p(\cdot)}(u)>0$, the norm
and the modular satisfy the inequality (see \cite[p75]{pxbook})
\begin{align}
\min\left\{ \varrho_{p(\cdot)}(u)^{\frac{1}{p^{-}}},\varrho_{p(\cdot)}(u)^{\frac{1}{p^{+}}}\right\}  & \leq\left\Vert u\right\Vert _{L^{p(\cdot)}(\Omega)}\leq\max\left\{ \varrho_{p(\cdot)}(u)^{\frac{1}{p^{-}}},\varrho_{p(\cdot)}(u)^{\frac{1}{p^{+}}}\right\} .\label{eq:modular_ineq}
\end{align}
A version of Hölder's inequality holds \cite[p81]{pxbook} : if $u\in L^{p(\cdot)}(\Omega)$
and $v\in L^{p^{\prime}(\cdot)}(\Omega)$, where $\frac{1}{p(x)}+\frac{1}{p^{\prime}(x)}=1$
for a.e. $x\in\Omega$, then 
\[
\int_{\Omega}\left|u\right|\left|v\right|\d x\leq2\left\Vert u\right\Vert _{L^{p(\cdot)}(\Omega)}\left\Vert v\right\Vert _{L^{p^{\prime}(\cdot)}(\Omega)}.
\]
As a consequence of the Hölder's inequality we have that 
\[
\left\Vert u\right\Vert _{L^{q(\cdot)}(\Omega)}\leq2\left(1+\left|\Omega\right|\right)\left\Vert u\right\Vert _{L^{p(\cdot)}(\Omega)}
\]
for all $u\in L^{p(\cdot)}(\Omega)$ if $q(x)\leq p(x)$ for a.e.
$x\in\Omega$. 

If $1<p^{-}\leq p^{+}<\infty$, then $L^{p(\cdot)}(\Omega)$ is reflexive
and the dual of $L^{p(\cdot)}(\Omega)$ is $L^{p^{\prime}(\cdot)}(\Omega)$. 

The \textit{variable exponent Sobolev space $W^{1,p(\cdot)}(\Omega)$}
is the set of functions in $u\in L^{p(\cdot)}(\Omega)$ for which
the weak gradient $Du$ belongs in $L^{p(\cdot)}(\Omega)$. It is
a Banach space equipped with the norm
\[
\left\Vert u\right\Vert _{W^{1,p(\cdot)}(\Omega)}:=\left\Vert u\right\Vert _{L^{p(\cdot)}(\Omega)}+\left\Vert Du\right\Vert _{L^{p(\cdot)}(\Omega)}.
\]

The space $W_{0}^{1,p}(\Omega)$ is the closure of compactly supported
Sobolev functions in the space $W^{1,p(\cdot)}(\Omega)$. A function
belongs to the the local Lebesgue space $L_{loc}^{p(\cdot)}(\Omega)$
if it belongs to $L^{p(\cdot)}(\Omega^{\prime})$ for all $\Omega^{\prime}\Subset\Omega$.
The local Sobolev space $W_{loc}^{1,p(\cdot)}(\Omega)$ is defined
analogically.

\section{The strong and normalized $p(x)$-Laplace equations}

In this section, we define weak solutions to the strong $p(x)$-Laplace
equation and viscosity solutions to the normalized $p(x)$-Laplace
equation.

From now on we assume that $p$ is Lipschitz continuous and $p^{-}>1$.
\begin{defn}
A function $u\in W_{loc}^{1,p(\cdot)}(\Omega)$ is a \textit{weak
supersolution} to $-\Delta_{p(x)}^{S}u\geq0$ in $\Omega$ if 
\[
\int_{\Omega}\left|Du\right|^{p(x)-2}Du\cdot D\varphi+\left|Du\right|^{p(x)-2}\log\left(\left|Du\right|\right)Du\cdot Dp\,\varphi\d x\geq0
\]
for all non-negative $\varphi\in W^{1,p(\cdot)}(\Omega)$ with compact
support. We say that $u$ is a \textit{weak subsolution} to $-\Delta_{p(x)}^{S}u\leq0$
if $-u$ is a supersolution and that $u$ is a \textit{weak solution}
to $-\Delta_{p(x)}^{S}u=0$ if $u$ is both supersolution and subsolution.
\end{defn}
\begin{lem}
\label{lem:cinfty_testfunctions}It is enough to consider $C_{0}^{\infty}(\Omega)$
test functions in the previous definition.
\end{lem}
\begin{proof}
Assume that $\varphi\in W^{1,p(\cdot)}(\Omega)$ has a compact support
in an open set $\Omega^{\prime}\Subset\Omega$. Since $p$ is log-Hölder
continuous and bounded as a Lipschitz function, there is a sequence
of functions $\varphi_{j}\in C_{0}^{\infty}(\Omega^{\prime})$ such
that $\varphi_{j}\rightarrow\varphi$ in $W^{1,p(\cdot)}(\Omega^{\prime})$
(see \cite[p347]{pxbook}). We set $\psi_{j}:=\varphi-\varphi_{j}$.
Then it is enough to show that
\[
\int_{\Omega^{\prime}}\left|Du\right|^{p(x)-2}Du\cdot D\psi_{j}\d x+\int_{\Omega^{\prime}}\left|Du\right|^{p(x)-2}\log\left(\left|Du\right|\right)Du\cdot Dp\,\psi_{j}\d x\rightarrow0
\]
as $j\rightarrow\infty$. The first integral convergences to zero
by Hölder's inequality so we focus on the second integral. We may
assume that $N>1$. We set $q(x):=\frac{p(x)}{p(x)-1+\frac{1}{N}}$.
Using the inequality $a^{s}\log a\leq Na^{s+\frac{1}{N}}+\frac{1}{s}$
for $a,s>0$ we get
\begin{align*}
\int_{\Omega^{\prime}} & \left|Du\right|^{p(x)-1}\left|\log\left|Du\right|\right|\left|Dp\right|\left|\psi_{j}\right|\d x\\
 & \leq\left\Vert Dp\right\Vert _{L^{\infty}(\Omega^{\prime})}\left(\int_{\Omega^{\prime}}\frac{\left|\psi_{j}\right|}{p(x)-1}\d x+N\int_{\Omega^{\prime}}\left|Du\right|^{p(x)-1+\frac{1}{N}}\left|\psi_{j}\right|\d x\right)\\
 & \leq C(p,\Omega)\left(\left\Vert \psi_{j}\right\Vert _{L^{p(\cdot)}(\Omega^{\prime})}+\left\Vert \left|Du\right|^{p(x)-1+\frac{1}{N}}\right\Vert _{L^{q(\cdot)}(\Omega^{\prime})}\left\Vert \psi_{j}\right\Vert _{L^{q^{\prime}(\cdot)}(\Omega^{\prime})}\right).
\end{align*}
We take $r\in(1,N)$ such that $q^{\prime+}\leq r^{\ast}:=\frac{Nr}{N-r}$.
Then we have $q^{\prime}(x)=\frac{Np(x)}{N-1}\leq\min(p^{\ast}(x),r^{\ast})$,
where $p^{\ast}(x):=\frac{Np(x)}{N-p(x)}$. Therefore
\[
\left\Vert \psi_{j}\right\Vert _{L^{q^{\prime}(\cdot)}(\Omega^{\prime})}\leq2\left(1+\left|\Omega\right|\right)\left\Vert \psi_{j}\right\Vert _{L^{\min(p^{\ast}(\cdot),r^{\ast})}(\Omega^{\prime})}.
\]
Since $\psi_{j}\in W_{0}^{1,\min(p(\cdot),r)}(\Omega^{\prime})$ ,
we have by a variable exponent version of the Sobolev inequality (see
e.g. \cite[p265]{pxbook})
\[
\left\Vert \psi_{j}\right\Vert _{L^{\min(p^{\ast}(\cdot),r^{\ast})}(\Omega^{\prime})}\leq C\left\Vert D\psi_{j}\right\Vert _{L^{\min(p(\cdot),r)}(\Omega^{\prime})}\leq2C(1+\left|\Omega\right|)\left\Vert D\psi_{j}\right\Vert _{L^{p(\cdot)}(\Omega^{\prime})}.
\]
These estimates imply the claim since $\left\Vert \psi_{j}\right\Vert _{W^{1,p}(\Omega^{\prime})}\rightarrow0$
as $j\rightarrow\infty$.
\end{proof}
In order to define viscosity solutions to $-\Delta_{p(x)}^{N}u=0$,
we set
\[
F(x,\eta,X):=-\left(\tr X+\frac{p(x)-2}{\left|\eta\right|^{2}}\left\langle X\eta,\eta\right\rangle \right)
\]
for all $(x,\eta,X)\in\Omega\times\left(\mathbb{R}^{N}\setminus\left\{ 0\right\} \right)\times S^{N}$
where $S^{N}$ is the set of symmetric $N\times N$ matrices. We also
recall the concept of semi-jets. The \textit{subjet of a function
$u:\Omega\rightarrow\mathbb{R}$ at} $x$ is defined by setting $(\eta,X)\in J^{2,-}u(x)$
if
\begin{equation}
u(y)\geq u(x)+\eta\cdot(y-x)+\frac{1}{2}\left\langle X(y-x),(y-x)\right\rangle +o(\left|y-x\right|^{2})\text{ as }y\rightarrow x.\label{eq:jet_ineq}
\end{equation}
The \textit{closure of a subjet} is defined by setting $(\eta,X)\in\overline{J}^{2,-}u(x)$
if there is a sequence $(\eta_{i},X_{i})\in J^{2,-}u(x_{i})$ such
that $(x_{i},\eta_{i},X_{i})\rightarrow(x,\eta,X)$. The \textit{superjet}
$J^{2,+}u(x)$ and its closure $\overline{J}^{2,+}u(x)$ are defined
in the same manner except that the inequality (\ref{eq:jet_ineq})
is reversed.
\begin{defn}
\label{def:viscositysuper}A lower semicontinuous function $u:\Omega\rightarrow\mathbb{R}$
is a \textit{viscosity supersolution} to $-\Delta_{p(x)}^{N}u\geq0$
in $\Omega$ if, whenever $(\eta,X)\in J^{2,-}u(x)$ with $x\in\Omega$
and $\eta\not=0$, then
\[
F(x,\eta,X)\geq0.
\]
A function $u$ is a \textit{viscosity subsolution} to $-\Delta_{p(x)}^{N}u\leq0$
if $-u$ is a viscosity supersolution, and a \textit{viscosity solution}
to $-\Delta_{p(x)}^{N}u=0$ if it is both viscosity super- and subsolution.
\end{defn}
\begin{rem*}
Observe that in the previous definition we require nothing in the
case $(0,X)\in J^{2,-}u(x)$.
\end{rem*}
Viscosity solutions may be equivalently defined using the jet-closures
or test functions. For the next proposition, see e.g. \cite[Prop 2.6]{koike}.
\begin{prop}
Let $u:\Omega\rightarrow\mathbb{R}$ be lower semicontinuous. Then
the following conditions are equivalent.
\begin{enumerate}
\item The function u is a viscosity supersolution to $-\Delta_{p(x)}^{N}u\geq0$
in $\Omega$.
\item Whenever $(\eta,X)\in\overline{J}^{2,-}u(x)$ with $x\in\Omega$,
$\eta\not=0$, we have $F(x,\eta,X)\geq0$.
\item Whenever $\varphi\in C^{2}(\Omega)$ is such that $\varphi(x)=u(x)$,
$D\varphi(x)\not=0$ and $\varphi(y)<u(y)$ for all $y\not=x$, it
holds $F(x,D\varphi(x),D^{2}\varphi(x))\geq0$.
\end{enumerate}
\end{prop}
When $\varphi$ is as in the third condition above, we say that $\varphi$
\textit{touches $u$ from below at} $x$.

\section{Weak solutions are Viscosity solutions}

We show that if $u$ is a weak solution to $-\Delta_{p(x)}^{S}u=0$,
then it is a viscosity solution to $-\Delta_{p(x)}^{N}u=0$. 

Juutinen, Lukkari and Parviainen \cite{equivalence_p(x)Laplace} showed
that weak solutions to the standard $p(x)$-Laplace equation are also
viscosity solutions. This was accomplished with the help of the comparison
principle. For if $u$ is a weak supersolution to $-\Delta_{p(x)}u\geq0$
that is not a viscosity supersolution, then there is a test function
$\varphi\in C^{2}$ touching $u$ from below at $x$ so that $-\Delta_{p(x)}\varphi<0$
in some ball $B(x)$. Lifting $\varphi$ slightly produces a new function
$\tilde{\varphi}$ still satisfying $-\Delta_{p(x)}\tilde{\varphi}<0$
in $B(x)$ and $\tilde{\varphi}\leq u$ in $\partial B(x)$. Comparison
principle now implies that $\tilde{\varphi}\leq u$ in $B(x)$ which
is a contradiction since $\tilde{\varphi}(x)>\varphi(x)=u(x)$.

Our difficulty is that, to the best of our knowledge, the comparison
principle is an open problem for the strong $p(x)$-Laplacian. Our
strategy is therefore to consider a ball so small that the gradient
of the test function does not vanish. Then the comparison principle
holds and we arrive at a contradiction.
\begin{thm}
\label{thm:visc}If $u\in W_{loc}^{1,p(\cdot)}(\Omega)$ is a weak
solution to $-\Delta_{p(x)}^{S}u=0$, then it is a viscosity solution
to $-\Delta_{p(x)}^{N}u=0$ in $\Omega$. 
\end{thm}
\begin{proof}
Zhang and Zhou \cite{strongpx_regularity} showed that weak solutions
of $-\Delta_{p(x)}^{S}u=0$ are in $C^{1}(\Omega)$. Therefore it
suffices to show that if $u\in C^{1}(\Omega)$ is a weak supersolution
to $-\Delta_{p(x)}^{S}u\geq0$, then it is also a viscosity supersolution
to $-\Delta_{p(x)}^{N}u\geq0$. Assume on the contrary that there
is $\varphi\in C^{2}(\text{\ensuremath{\Omega})}$ touching $u$ from
below at $x_{0}\in\Omega$, $D\varphi(x_{0})\not=0$ and
\[
0>-h>F(x_{0},D\varphi(x_{0}),D^{2}\varphi(x_{0})).
\]
Then by continuity there is $r>0$ such that in $B_{r}(x_{0})$ it
holds
\begin{align}
-h\left|D\varphi\right|^{p(x)-2}\geq & -\left|D\varphi\right|^{p(x)-2}\left(\Delta\varphi+\frac{p(x)-2}{\left|D\varphi\right|^{2}}\Delta_{\infty}\varphi\right).\label{eq:visc_ineq1}
\end{align}
Since $Du(x_{0})=D\varphi(x_{0})\not=0$, we may also assume that
there is $m>0$ such that
\begin{equation}
\inf_{x\in B_{r}(x_{0})}\left|D\varphi\right|^{p(x)-2}\geq m\label{eq:visc_ineq22}
\end{equation}
and
\begin{equation}
\underset{x\in B_{r}(x_{0})}{\mathrm{ess\,sup}}\left|Dp\right|\left|\left|D\varphi\right|^{p(x)-2}\log\left(\left|D\varphi\right|\right)D\varphi-\left|Du\right|^{p(x)-2}\log\left(\left|Du\right|\right)Du\right|\leq\frac{hm}{2}.\label{eq:visc_ineq33}
\end{equation}
Let $l:=\min_{x\in\partial B_{r}(x_{0})}\left(u-\varphi\right)>0$
and set $\psi(x):=\max\left(\varphi(x)+l-u(x),0\right).$ Then $\psi\in W_{0}^{1,2}(B_{r}(x_{0}))$
so there are $\psi_{j}\in C_{0}^{\infty}(B_{r}(x_{0}))$ such that
$\psi_{j}\rightarrow\psi$ in $W^{1,2}(B_{r}(x_{0}))$. Let $p_{j}$
be the standard mollification of $p$. Multiplying (\ref{eq:visc_ineq1})
by $\psi$ and integrating over $B_{r}(x_{0})$ yields
\begin{align}
-h\int_{B_{r}(x_{0})} & \left|D\varphi\right|^{p(x)-2}\psi\d x\nonumber \\
\geq & \int_{B_{r}(x_{0})}-\left|D\varphi\right|^{p(x)-2}\left(\Delta\varphi+\frac{p(x)-2}{\left|D\varphi\right|^{2}}\Delta_{\infty}\varphi\right)\psi\d x\nonumber \\
= & \lim_{j\rightarrow\infty}\int_{B_{r}(x_{0})}-\left|D\varphi\right|^{p_{j}(x)-2}\left(\Delta\varphi+\frac{p_{j}(x)-2}{\left|D\varphi\right|^{2}}\Delta_{\infty}\varphi\right)\psi_{j}\d x,\label{eq:visc_ineq11}
\end{align}
where the last equality holds because $\psi_{j}\rightarrow\psi$ in
$W^{1,2}(B_{r}(x_{0}))$ and $p_{j}\rightarrow p$ uniformly in $B_{r}(x_{0})$.
Calculating the divergence of $\left|D\varphi\right|^{p_{j}(x)-2}D\varphi$
and integrating by parts we get
\begin{align}
\int_{B_{r}(x_{0})} & -\left|D\varphi\right|^{p_{j}(x)-2}\left(\Delta\varphi+\frac{p_{j}(x)-2}{\left|D\varphi\right|^{2}}\Delta_{\infty}\varphi\right)\psi_{j}\d x\nonumber \\
= & \int_{B_{r}(x_{0})}-\mathrm{div}\left(\left|D\varphi\right|^{p_{j}(x)-2}D\varphi\right)\psi_{j}+\left|D\varphi\right|^{p_{j}(x)-2}\log\left(\left|D\varphi\right|\right)D\varphi\cdot Dp_{j}\,\psi_{j}\d x\nonumber \\
= & \int_{B_{r}(x_{0})}\left|D\varphi\right|^{p_{j}(x)-2}D\varphi\cdot\left(D\psi_{j}+\log\left(\left|D\varphi\right|\right)Dp_{j}\,\psi_{j}\right)\d x.\label{eq:visc_eq22}
\end{align}
By the convergence of $\psi_{j}$ and $p_{j}$, it follows from (\ref{eq:visc_ineq11})
and (\ref{eq:visc_eq22}) that
\begin{align}
- & h\int_{B_{r}(x_{0})}\left|D\varphi\right|^{p(x)-2}\psi\d x\geq\int_{B_{r}(x_{0})}\left|D\varphi\right|^{p(x)-2}D\varphi\cdot\left(D\psi+\log\left(\left|D\varphi\right|\right)Dp\,\psi\right)\d x.\label{eq:visc_ineq2}
\end{align}
Since $u$ is a weak supersolution to $\Delta_{p(x)}^{S}u=0$ and
$\psi\in W^{1,p(\cdot)}(\Omega)$ has a compact support in $\Omega$,
we have 
\begin{equation}
\int_{B_{r}(x_{0})}\left|Du\right|^{p(x)-2}Du\cdot\left(D\psi+\log\left|Du\right|Dp\,\psi\right)\d x\geq0.\label{eq:visc_eq3}
\end{equation}
Denoting $A:=\left\{ x\in B_{r}(x_{0}):\psi(x)>0\right\} $ and combining
(\ref{eq:visc_ineq2}) and (\ref{eq:visc_eq3}) we arrive at
\begin{align}
\int_{A} & \left(\left|D\varphi\right|^{p(x)-2}D\varphi-\left|Du\right|^{p(x)-2}Du\right)\cdot\left(D\varphi-Du\right)\d x\nonumber \\
\leq & \int_{A}\left|\left|Du\right|^{p(x)-2}\log\left(\left|Du\right|\right)Du-\left|D\varphi\right|^{p(x)-2}\log\left(\left|D\varphi\right|\right)D\varphi\right|\left|Dp\right|\psi\d x\nonumber \\
 & -h\int_{A}\left|D\varphi\right|^{p(x)-2}\psi\d x\nonumber \\
\leq & -\frac{hm}{2}\int_{A}\psi\d x,\label{eq:visc_ineq4}
\end{align}
where the last inequality follows from (\ref{eq:visc_ineq22}) and
(\ref{eq:visc_ineq33}). Since 
\[
\left(\left|a\right|^{p(x)-2}a-\left|b\right|^{p(x)-2}b\right)\cdot\left(a-b\right)\geq0
\]
for any two vectors $a,b\in\mathbb{R}^{N}$ when $p(x)>1$, it follows
from (\ref{eq:visc_ineq4}) that $\left|A\right|=0$. But this is
impossible since $\varphi(x_{0})=u(x_{0})$ and $l>0$. 
\end{proof}

\section{Viscosity solutions are Weak solutions}

We show that if $u$ is a viscosity supersolution to $-\Delta_{p(x)}^{N}u\geq0$,
then it is a weak supersolution to $-\Delta_{p(x)}^{S}u\geq0$. The
same statement for subsolutions then follows by analogy.

We recall the usual partial ordering for symmetric $N\times N$ matrices
by setting $X\leq Y$ if $\left\langle X\xi,\xi\right\rangle \leq\left\langle Y\xi,\xi\right\rangle $
for all $\xi\in\mathbb{R}^{N}$. For a matrix $X$ we also set $\left\Vert X\right\Vert :=\max\left\{ \left|\lambda\right|:\lambda\text{ is an eigenvalue of }X\right\} $
and for vectors $\xi,\eta\in\mathbb{R}^{N}$ we use the notation $\xi\otimes\eta:=\xi\eta^{\prime}$,
i.e. $\xi\otimes\eta$ is an $N\times N$ matrix whose $\left(i,j\right)$
entry is $\xi_{i}\eta_{j}$.
\begin{defn}[Inf-convolution]
Let $q\geq2$ and $\varepsilon>0$. The inf-convolution of a bounded
function $u\in C(\Omega)$ is defined by
\begin{equation}
u_{\varepsilon}(x):=\inf_{y\in\Omega}\left\{ u(y)+\frac{1}{q\varepsilon^{q-1}}\left|x-y\right|^{q}\right\} .\label{eq:inf_convolution}
\end{equation}
\end{defn}
The inf-convolution is well known to provide good approximations of
viscosity supersolutions and often one only needs to consider it for
$q=2$ (see e.g. \cite{userguide}). However, as the authors in \cite{newequivalence}
observed, considering large enough $q$ essentially cancels the singularity
in the usual $p$-Laplace operator when $1<p<2$. In similar fashion
it also cancels the singularity of the operator $\Delta_{p(x)}^{S}$.
This is due to the property (v) in the next lemma. We also list some
other basic properties of the inf-convolution.
\begin{lem}
\label{lem:infconv_properties}Let $u\in C(\Omega)$ be a bounded
function. Then the inf-convolution $u_{\varepsilon}$ as defined in
(\ref{eq:inf_convolution}) has the following properties.
\begin{enumerate}
\item We have $u_{\varepsilon}\leq u$ in $\Omega$ and $u_{\varepsilon}\rightarrow u$
locally uniformly in $\Omega$ as $\varepsilon\rightarrow0$.
\item There exists $r(\varepsilon)>0$ such that
\[
u_{\varepsilon}(x)=\inf_{y\in B_{r(\varepsilon)}(x)\cap\Omega}\left\{ u(y)+\frac{1}{q\varepsilon^{q-1}}\left|x-y\right|^{q}\right\} 
\]
and $r(\varepsilon)\rightarrow0$ as $\varepsilon\rightarrow0$. In
fact we can choose $r(\varepsilon)=\left(q\varepsilon^{q-1}\mathrm{osc}_{\Omega}\thinspace u\right)^{\frac{1}{q}}$.
\item The function $u_{\varepsilon}$ is semi-concave in $\Omega_{r(\varepsilon)}$,
that is, the function $x\mapsto u_{\varepsilon}(x)-\frac{q-1}{2\varepsilon^{q-1}}r(\varepsilon)^{q-2}\left|x\right|^{2}$
is concave. 
\item If $x\in\Omega_{r(\varepsilon)}:=\left\{ x\in\Omega:\mathrm{dist}(x,\partial\Omega)<r(\varepsilon)\right\} $,
then there exists a point $x_{\varepsilon}\in B_{r(\varepsilon)}(x)$
such that $u_{\varepsilon}(x)=u(x_{\varepsilon})+\frac{1}{q\varepsilon^{q-1}}\left|x-x_{\varepsilon}\right|^{q}$. 
\item If $(\eta,X)\in J^{2,-}u_{\varepsilon}(x)$ with $x\in\Omega_{r(\varepsilon)}$,
then $\eta=\frac{\left(x-x_{\varepsilon}\right)}{\varepsilon^{q-1}}\left|x_{\varepsilon}-x\right|^{q-2}$
and $X\leq\frac{q-1}{\varepsilon}\left|\eta\right|^{\frac{q-2}{q-1}}I$,
where $x_{\varepsilon}$ is as in \textup{(iv)}.
\end{enumerate}
\end{lem}
These properties are well known, see appendix of \cite{newequivalence}
and also \cite{convex_functionals} where more general ``flat inf-convolution''
is considered\textit{.} Regardless, we give a proof of (v) based on
\cite[p53]{nikos} due to its critical role in the proof of Lemma
\ref{lem:proof_lemma}.
\begin{proof}[Proof of property (v) in Lemma \ref{lem:infconv_properties}]
Let $(\eta,X)\in J^{2,-}u_{\varepsilon}(x)$. Then there is a function
$\varphi\in C^{2}(\mathbb{R}^{N})$ such that it touches $u_{\varepsilon}$
from below at $x$ and $D\varphi(x)=\eta$, $D^{2}\varphi(x)=X$.
Therefore for all $y,z\in\Omega$ we have
\begin{align*}
u(y)+\frac{\left|y-z\right|^{q}}{q\varepsilon^{q-1}}-\varphi(z)\geq & u_{\varepsilon}(z)-\varphi(z)\geq0.
\end{align*}
Choosing $y=x_{\varepsilon}$, we obtain
\[
\varphi(z)-\frac{\left|x_{\varepsilon}-z\right|^{q}}{q\varepsilon^{q-1}}\leq u(x_{\varepsilon})\text{ for all }z\in\Omega.
\]
Since $\varphi(x)=u_{\varepsilon}(x)=u(x_{\varepsilon})+\frac{\left|x_{\varepsilon}-x\right|^{q}}{q\varepsilon^{q-1}}$,
the above inequality means that the function
\[
z\mapsto\varphi(z)-\frac{\left|x_{\varepsilon}-z\right|^{q}}{q\varepsilon^{q-1}}=:\varphi(z)-\psi(z)
\]
has a maximum at $x$. Thus $\eta=D\psi(x)=\frac{\left(x-x_{\varepsilon}\right)}{\varepsilon^{q-1}}\left|x_{\varepsilon}-x\right|^{q-2}$
and
\begin{align*}
X\leq D^{2}\psi(x)= & \frac{1}{\varepsilon^{q-1}}\left|x_{\varepsilon}-x\right|^{q-4}\left(\left(q-2\right)\left(x_{\varepsilon}-x\right)\otimes\left(x_{\varepsilon}-x\right)+\left|x_{\varepsilon}-x\right|^{2}I\right)\\
\leq & \frac{1}{\varepsilon^{q-1}}\left|x_{\varepsilon}-x\right|^{q-4}\left(\left(q-2\right)\left\Vert \left(x_{\varepsilon}-x\right)\otimes\left(x_{\varepsilon}-x\right)\right\Vert I+\left|x_{\varepsilon}-x\right|^{2}I\right)\\
= & \frac{q-1}{\varepsilon^{q-1}}\left|x_{\varepsilon}-x\right|^{q-2}I\\
= & \frac{q-1}{\varepsilon^{q-1}}\left(\varepsilon\left|\eta\right|^{\frac{1}{q-1}}\right)^{q-2}I\\
= & \frac{q-1}{\varepsilon}\left|\eta\right|^{\frac{q-2}{q-1}}I.\qedhere
\end{align*}
\end{proof}
We will show that the inf-convolution provides approximations of viscosity
supersolutions to $-\Delta_{p(x)}^{N}u\geq0$. If there was no $x$-dependence
in the equation, it would be straightforward to show that the inf-convolution
of a supersolution is still a supersolution. However, the equation
$-\Delta_{p(x)}^{N}u\geq0$ has $x$-dependence caused by $p(x)$.
Regardless, in \cite[Thm 3]{Ishii95} it is shown that with some assumptions
on $G$, the inf-convolution $u_{\varepsilon}$ of a viscosity supersolution
to $G(x,u,Du,D^{2}u)\geq0$ is still a viscosity supersolution to
$G(x,u_{\varepsilon},Du_{\varepsilon},D^{2}u_{\varepsilon})\geq E(\varepsilon)$,
where $E(\varepsilon)\rightarrow0$ as $\varepsilon\rightarrow0$.

We prove a modified version of this theorem for the solutions of $-\Delta_{p(x)}^{N}u\geq0$.
The important modification is the term $\left|\eta\right|^{\min(p(x)-2,0)}$
in (\ref{eq:ishiiq_claimineq}) as it cancels a singular gradient
term that appears due to the error term in the proof of Lemma \ref{lem:proof_lemma},
see (\ref{eq:proof_lemma_liminfstuff}). Another difference is that
we consider inf-convolution with the exponent $q\geq2$.
\begin{lem}
\label{lem:ishiiq}Assume that $u$ is a uniformly continuous viscosity
supersolution to $-\Delta_{p(x)}^{N}u\geq0$ in $\Omega$. Then, whenever
$(\eta,X)\in J^{2,-}u_{\varepsilon}(x)$, $\eta\not=0$ and $x\in\Omega_{r(\varepsilon)}$,
it holds
\begin{equation}
\left|\eta\right|^{\min(p(x)-2,0)}F(x,\eta,X)\geq E(\varepsilon),\label{eq:ishiiq_claimineq}
\end{equation}
where $E(\varepsilon)\rightarrow0$ as $\varepsilon\rightarrow0$.
The error function $E$ depends only on $p$, $q$ and the modulus
of continuity of $u$.
\end{lem}
\begin{proof}
Fix $x\in\Omega_{r(\varepsilon)}$ and $(\eta,X)\in J^{2,-}u_{\varepsilon}(x)$,
$\eta\not=0$. Then by Lemma \ref{lem:infconv_properties} there is
$x_{\varepsilon}\in B_{r(\varepsilon)}(x)$ such that 
\begin{equation}
u_{\varepsilon}(x)=u(x_{\varepsilon})+\frac{\left|x_{\varepsilon}-x\right|^{q}}{q\varepsilon^{q-1}}\label{eq:ishiiq_xe}
\end{equation}
and $\eta=\frac{\left(x-x_{\varepsilon}\right)}{\varepsilon^{q-1}}\left|x_{\varepsilon}-x\right|^{q-2}$.
There exists a function $\varphi\in C^{2}(\mathbb{R}^{N})$ such that
it touches $u_{\varepsilon}$ from below at $x$ and $D\varphi(x)=\eta$,
$D^{2}\varphi(x)=X$. By the definition of inf-convolution
\begin{align}
u(y)-\varphi(z)+ & \frac{\left|y-z\right|^{q}}{q\varepsilon^{q-1}}\geq u_{\varepsilon}(z)-\varphi(z)\geq0\text{ for all }y,z\in\Omega_{r(\varepsilon)}.\label{eq:ishiiqq1}
\end{align}
Since by (\ref{eq:ishiiq_xe}) we have $u(x_{\varepsilon})=\varphi(x)-\frac{\left|x_{\varepsilon}-x\right|^{q}}{q\varepsilon^{q-1}},$
it follows from (\ref{eq:ishiiqq1}) that the expression $u(y)-\varphi(z)+\frac{\left|y-z\right|^{q}}{q\varepsilon^{q-1}}$
reaches its minimum at $(y,z)=(x_{\varepsilon},x)$. Thus 
\begin{align*}
\max_{(y,z)\in\Omega_{r(\varepsilon)}\times\Omega_{r(\varepsilon)}}-u(y)+\varphi(z)-\frac{\left|y-z\right|^{q}}{q\varepsilon^{q-1}}=- & u(x_{\varepsilon})+\varphi(x)-\frac{\left|x_{\varepsilon}-x\right|^{q}}{q\varepsilon^{q-1}}.
\end{align*}
\[
\]
We denote $\Phi(y,z):=\frac{1}{q\varepsilon^{q-1}}\left|y-z\right|^{q}$
and invoke the Theorem of sums (see \cite{userguide}). There exist
$Y,Z\in S^{N}$ such that 
\[
(\eta,-Y)\in\overline{J}^{2,-}u(x_{\varepsilon}),\thickspace\thickspace(\eta,-Z)\in\overline{J}^{2,+}\varphi(x)
\]
and
\begin{equation}
\begin{pmatrix}Y & 0\\
0 & -Z
\end{pmatrix}\leq D^{2}\Phi(x_{\varepsilon},x)+\varepsilon^{q-1}\left(D^{2}\Phi(x_{\varepsilon},x)\right)^{2}\label{eq:ishiiq_sumineq}
\end{equation}
where
\[
D^{2}\Phi(x_{\varepsilon},x)=\left(\begin{array}{cc}
M & -M\\
-M & M
\end{array}\right)
\]
with $M=\frac{1}{\varepsilon^{q-1}}\left|x_{\varepsilon}-x\right|^{q-4}\left(\left(q-2\right)\left(x_{\varepsilon}-x\right)\otimes\left(x_{\varepsilon}-x\right)+\left|x_{\varepsilon}-x\right|^{2}I\right)$
and
\[
\left(D^{2}\Phi(x_{\varepsilon},x)\right)^{2}=2\begin{pmatrix}M^{2} & -M^{2}\\
-M^{2} & M^{2}
\end{pmatrix}.
\]
The above implies $Y\leq Z\leq-D^{2}\varphi(x)=-X$. Multiplying (\ref{eq:ishiiq_sumineq})
by the $\mathbb{R}^{2N}$ vector $(\frac{\eta}{\left|\eta\right|}\sqrt{p(x_{\varepsilon})-1},\frac{\eta}{\left|\eta\right|}\sqrt{p(x)-1})$
from both sides yields
\begin{equation}
\frac{(p(x_{\varepsilon})-1)}{\left|\eta\right|^{2}}\left\langle Y\eta,\eta\right\rangle -\frac{\left(p(x)-1\right)}{\left|\eta\right|^{2}}\left\langle Z\eta,\eta\right\rangle \leq\varLambda^{2}\left\langle \left(M+2\varepsilon^{q-1}M^{2}\right)\frac{\eta}{\left|\eta\right|},\frac{\eta}{\left|\eta\right|}\right\rangle ,\label{eq:ishiiq_lambdaineq}
\end{equation}
where $\varLambda=\sqrt{p(x)-1}-\sqrt{p(x_{\varepsilon})-1}$. We
have
\begin{align}
0\leq & F(x_{\varepsilon},\eta,-Y)\nonumber \\
= & F(x,\eta,Z)-F(x_{\varepsilon},\eta,Y)-F(x,\eta,Z)\nonumber \\
= & \left(p(x_{\varepsilon})-1\right)\left\langle Y\frac{\eta}{\left|\eta\right|},\frac{\eta}{\left|\eta\right|}\right\rangle -\left(p(x)-1\right)\left\langle Z\frac{\eta}{\left|\eta\right|},\frac{\eta}{\left|\eta\right|}\right\rangle \nonumber \\
 & +\tr(Y)-\left\langle Y\frac{\eta}{\left|\eta\right|},\frac{\eta}{\left|\eta\right|}\right\rangle -\tr(Z)+\left\langle Z\frac{\eta}{\left|\eta\right|},\frac{\eta}{\left|\eta\right|}\right\rangle +F(x,\eta,-Z)\nonumber \\
\leq & \varLambda^{2}\left\langle \left(M+2\varepsilon^{q-1}M^{2}\right)\frac{\eta}{\left|\eta\right|},\frac{\eta}{\left|\eta\right|}\right\rangle +F(x,\eta,X),\label{eq:ishiiq_main}
\end{align}
where we used (\ref{eq:ishiiq_lambdaineq}) and the fact that $Y\leq Z$
implies
\[
\tr\left(Y-Z\right)-\left\langle \left(Y-Z\right)\frac{\eta}{\left|\eta\right|},\frac{\eta}{\left|\eta\right|}\right\rangle \leq0.
\]
We have the estimate
\begin{align*}
\left\Vert M\right\Vert \leq & \frac{1}{\varepsilon^{q-1}}\left|x_{\varepsilon}-x\right|^{q-4}\left(\left(q-2\right)\left\Vert \left(x_{\varepsilon}-x\right)\otimes\left(x_{\varepsilon}-x\right)\right\Vert +\left|x_{\varepsilon}-x\right|^{2}\left\Vert I\right\Vert \right)\\
= & \frac{q-1}{\varepsilon^{q-1}}\left|x_{\varepsilon}-x\right|^{q-2}.
\end{align*}
Since $p$ is Lipschtiz continuous and $p^{-}>1$, we have also
\begin{align*}
\varLambda^{2}= & \frac{\left|p(x)-p(x_{\varepsilon})\right|^{2}}{\left|\sqrt{p(x)-1}+\sqrt{p(x_{\varepsilon})-1}\right|^{2}}\leq C(p)\left|x-x_{\varepsilon}\right|^{2}.
\end{align*}
Combining these with (\ref{eq:ishiiq_main}) we get (we may assume
that $r(\varepsilon)<1$)
\begin{align}
-F(x,\eta,X)\leq & \varLambda^{2}\left(\left\Vert M\right\Vert +2\varepsilon^{q-1}\left\Vert M\right\Vert ^{2}\right)\nonumber \\
\leq & \varLambda^{2}\left(\frac{q-1}{\varepsilon^{q-1}}\left|x_{\varepsilon}-x\right|^{q-2}+2\varepsilon^{q-1}\left(\frac{q-1}{\varepsilon^{q-1}}\right)^{2}\left|x_{\varepsilon}-x\right|^{2(q-2)}\right)\nonumber \\
\leq & \frac{3\left(q-1\right)^{2}}{\varepsilon^{q-1}}\varLambda^{2}\left|x_{\varepsilon}-x\right|^{q-2}\nonumber \\
\leq & C(p,q)\frac{1}{\varepsilon^{q-1}}\left|x_{\varepsilon}-x\right|^{q}.\label{eq:ishiiq_main2}
\end{align}
Moreover, by uniform continuity of $u$ there is a modulus of continuity
$\omega$ such that $\omega(t)\rightarrow0$ as $t\rightarrow0$ and
$\left|u(y)-u(z)\right|\leq\omega(\left|y-z\right|)$ for all $y,z\in\Omega$.
Hence by (\ref{eq:ishiiq_xe})
\begin{align}
\left|x_{\varepsilon}-x\right|\leq & \left(q\varepsilon^{q-1}\left(u(x)-u(x_{\varepsilon})\right)\right)^{\frac{1}{q}}\leq q^{\frac{1}{q}}\varepsilon^{\frac{q-1}{q}}\omega(r(\varepsilon))^{\frac{1}{q}}.\label{eq:ishiiq_uniformest}
\end{align}
We now consider the situations $p(x)\leq2$ and $p(x)>2$ separately.\\
If $p(x)\leq2$, we multiply (\ref{eq:ishiiq_main2}) by $\left|\eta\right|^{p(x)-2}$
and estimate using (\ref{eq:ishiiq_uniformest}). We get
\begin{align*}
-\left|\eta\right|^{p(x)-2}F(x,\eta,X)\leq & C(p,q)\frac{1}{\varepsilon^{q-1}}\left|x_{\varepsilon}-x\right|^{q}\left|\eta\right|^{p(x)-2}\\
= & C(p,q)\frac{1}{\varepsilon^{q-1}}\left|x_{\varepsilon}-x\right|^{q}\left|\frac{1}{\varepsilon^{q-1}}\left(x-x_{\varepsilon}\right)\left|x_{\varepsilon}-x\right|^{q-2}\right|^{p(x)-2}\\
= & C(p,q)\left(\frac{1}{\varepsilon}\right)^{(q-1)(p(x)-1)}\left|x_{\varepsilon}-x\right|^{q+(q-1)(p(x)-2)}\\
\leq & C(p,q)\left(\frac{1}{\varepsilon}\right)^{(q-1)(p(x)-1)}\left(q^{\frac{1}{q}}\varepsilon^{\frac{q-1}{q}}\omega(r(\varepsilon))^{\frac{1}{q}}\right)^{q+(q-1)(p(x)-2)}\\
= & C(p,q)\left(\frac{1}{\varepsilon}\right)^{\left(\frac{q-1}{q}\right)(p(x)-2)}\omega(r(\varepsilon))^{\frac{q+(q-1)(p(x)-2)}{q}}\\
\leq & C(p,q)\omega(r(\varepsilon))^{\frac{q+(q-1)(p^{-}-2)}{q}},
\end{align*}
where the last inequality is true when $\varepsilon<1$ is so small
that $\omega(r(\varepsilon))<1$. This proves (\ref{eq:ishiiq_claimineq})
when $p(x)\leq2$.\\
If $p(x)>2$, we estimate (\ref{eq:ishiiq_main2}) directly using
(\ref{eq:ishiiq_uniformest}). We get
\begin{align*}
-F(x,\eta,X)\leq & C(p,q)\frac{1}{\varepsilon^{q-1}}\left(q^{\frac{1}{q}}\varepsilon^{\frac{q-1}{q}}\omega(r(\varepsilon))^{\frac{1}{q}}\right)^{q}=C(p,q)\omega(r(\varepsilon))),
\end{align*}
which proves (\ref{eq:ishiiq_claimineq}) when $p(x)>2$.
\end{proof}
Next we will use the previous lemma to show that inf-convolution of
a viscosity supersolution to $-\Delta_{p(x)}^{N}u\geq0$ in $\Omega$
is a weak supersolution to $-\Delta_{p(x)}^{S}u\geq0$ in $\Omega_{r(\varepsilon)}$
up to some error term. Before proceeding we make some remarks about
the point-wise differentiability of inf-convolution.
\begin{rem}
\label{rem:inf_remark}It follows from semi-concavity that the inf-convolution
$u_{\varepsilon}$ is locally Lipschitz in $\Omega_{r(\varepsilon)}$
(see \cite[p267]{measuretheoryevans}). Therefore it belongs in $W_{loc}^{1,\infty}(\Omega_{r(\varepsilon)})$,
is differentiable almost everywhere in $\Omega_{r(\varepsilon)}$,
and its derivative agrees with its Sobolev derivative almost everywhere
in $\Omega_{r(\varepsilon)}$ (see \cite[p155 and p265]{measuretheoryevans}). 

By Lemma \ref{lem:infconv_properties} the function $\phi(x):=u_{\varepsilon}(x)-C(q,\varepsilon,u)\left|x\right|^{2}$
is concave in $\Omega_{r(\varepsilon)}$. Thus Alexandrov's theorem
implies that $u_{\varepsilon}$ is twice differentiable almost everywhere
in $\Omega_{r(\varepsilon)}$. Furthermore, the proof of Alexandrov's
theorem in \cite[p273]{measuretheoryevans} establishes that if $\phi_{j}$
is the standard mollification of $\phi$, then $D^{2}\phi_{j}\rightarrow D^{2}\phi$
almost everywhere in $\Omega_{r(\varepsilon)}$.
\end{rem}
\begin{lem}
\label{lem:proof_lemma} Assume that $u$ is a uniformly continuous
viscosity supersolution to $-\Delta_{p(x)}^{N}u\geq0$ in $\Omega$.
Let $q>2$ be so large that $p^{-}-2+\frac{q-2}{q-1}\geq0$ and let
$u_{\varepsilon}$ be the inf-convolution of $u$ as defined in (\ref{eq:inf_convolution}).
Then 
\[
\int_{\Omega_{r(\varepsilon)}}\left|Du_{\varepsilon}\right|^{p(x)-2}Du_{\varepsilon}\cdot\left(D\varphi+\log\left|Du_{\varepsilon}\right|Dp\,\varphi\right)\d x\geq E(\varepsilon)\int_{\Omega_{r(\varepsilon)}}\left|Du_{\varepsilon}\right|^{s(x)}\varphi\d x
\]
for all non-negative $\varphi\in W^{1,p(\cdot)}(\Omega_{r(\varepsilon)})$
with compact support, where $E(\varepsilon)\rightarrow0$ as $\varepsilon\rightarrow0$
and $s(x)=\max(p(x)-2,0)$.
\end{lem}
\begin{proof}
It is enough to consider $\varphi\in C_{0}^{\infty}(\Omega_{r(\varepsilon)})$.
This can be seen in the same way as Lemma \ref{lem:cinfty_testfunctions},
but since $u_{\varepsilon}\in W_{loc}^{1,\infty}(\Omega_{r(\varepsilon)})$,
the proof is even simpler.

\textbf{(Step 1)} We show that $u_{\varepsilon}$ satisfies the auxiliary
inequality (\ref{eq:aux}) for all $0<\delta<1$. As mentioned in
Remark \ref{rem:inf_remark}, the function $\phi(x):=u_{\varepsilon}(x)-C(q,\varepsilon,u)\left|x\right|^{2}$
is concave in $\Omega_{r(\varepsilon)}$ and we can approximate it
by smooth concave functions $\phi_{j}$ so that $\left(\phi_{j},D\phi_{j},D^{2}\phi_{j}\right)\rightarrow\left(\phi,D\phi,D^{2}\phi\right)$
almost everywhere in $\Omega_{r(\varepsilon)}$. We define 
\[
u_{\varepsilon,j}(x):=\phi_{j}(x)+C(q,\varepsilon,u)\left|x\right|^{2}
\]
and denote by $p_{j}$ the standard mollification of $p$. Since $u_{\varepsilon,j}$
and $p_{j}$ are smooth, we calculate
\begin{align}
\int_{\Omega_{r(\varepsilon)}} & -\Big(\delta+\left|Du_{\varepsilon,j}\right|^{2}\Big)^{\frac{p_{j}(x)-2}{2}}\left(\Delta u_{\varepsilon,j}+\frac{p_{j}(x)-2}{\delta+\left|Du_{\varepsilon,j}\right|^{2}}\Delta_{\infty}u_{\varepsilon,j}\right)\varphi\d x\nonumber \\
= & \int_{\Omega_{r(\varepsilon)}}-\mathrm{div}\left(\left(\delta+\left|Du_{\varepsilon,j}\right|^{2}\right)^{\frac{p_{j}(x)-2}{2}}Du_{\varepsilon,j}\right)\varphi\nonumber \\
 & +\frac{1}{2}\left(\delta+\left|Du_{\varepsilon,j}\right|^{2}\right)^{\frac{p_{j}(x)-2}{2}}\log\left(\delta+\left|Du_{\varepsilon,j}\right|^{2}\right)Du_{\varepsilon,j}\cdot Dp_{j}\,\varphi\d x\nonumber \\
= & \int_{\Omega_{r(\varepsilon)}}\left(\delta+\left|Du_{\varepsilon,j}\right|^{2}\right)^{\frac{p_{j}(x)-2}{2}}Du_{\varepsilon,j}\cdot\left(D\varphi+\frac{1}{2}\log\left(\delta+\left|Du_{\varepsilon,j}\right|^{2}\right)Dp_{j}\,\varphi\right)\d x.\label{eq:jaux}
\end{align}
We let $j\rightarrow\infty$ in (\ref{eq:jaux}) and intend to use
Fatou's lemma at the LHS and the Dominated convergence theorem at
the RHS. This results in the auxiliary inequality
\begin{align}
\int_{\Omega_{r(\varepsilon)}} & -\Big(\delta+\left|Du_{\varepsilon}\right|^{2}\Big)^{\frac{p(x)-2}{2}}\left(\Delta u_{\varepsilon}+\frac{p(x)-2}{\delta+\left|Du_{\varepsilon}\right|^{2}}\Delta_{\infty}u_{\varepsilon}\right)\varphi\d x\nonumber \\
\leq & \int_{\Omega_{r(\varepsilon)}}\left(\delta+\left|Du_{\varepsilon}\right|^{2}\right)^{\frac{p(x)-2}{2}}Du_{\varepsilon}\cdot\left(D\varphi+\frac{1}{2}\log\left(\delta+\left|Du_{\varepsilon}\right|^{2}\right)Dp\,\varphi\right)\d x,\label{eq:aux}
\end{align}
where $D^{2}u_{\varepsilon}$ is the Hessian of $u_{\varepsilon}$
in the Alexandrov's sense. We still need to check that the assumptions
of the Dominated convergence theorem and Fatou's lemma hold. By Lipschitz
continuity of $u_{\varepsilon}$ and $p$ there is $M\geq1$ such
that
\[
\sup_{j}\left\Vert Du_{\varepsilon,j}\right\Vert _{L^{\infty}(\supp\varphi)},\sup_{j}\left\Vert Dp_{j}\right\Vert _{L^{\infty}(\supp\varphi)}\leq M.
\]
This justifies our use of the Dominated convergence theorem. In order
to justify our use of Fatou's lemma, we notice first that by concavity
of $\phi_{j}$ we have $D^{2}u_{\varepsilon,j}\leq C(q,\varepsilon,u)I$.
Thus the integrand at the LHS of (\ref{eq:jaux}) is clearly bounded
from below by a constant independent of $j$ if $Du_{\varepsilon,j}=0$.
If $Du_{\varepsilon,j}\not=0$, we have
\begin{align*}
\bigg(\delta & +\left|Du_{\varepsilon,j}\right|^{2}\bigg)^{\frac{p_{j}(x)-2}{2}}\left(\Delta u_{\varepsilon,j}+\frac{p_{j}(x)-2}{\delta+\left|Du_{\varepsilon,j}\right|^{2}}\Delta_{\infty}u_{\varepsilon,j}\right)\\
 & =\frac{\left(\delta+\left|Du_{\varepsilon,j}\right|^{2}\right)^{\frac{p_{j}(x)-2}{2}}}{\delta+\left|Du_{\varepsilon,j}\right|^{2}}\left(\left|Du_{\varepsilon,j}\right|^{2}\left(\Delta u_{\varepsilon,j}+\frac{p_{j}(x)-2}{\left|Du_{\varepsilon,j}\right|^{2}}\Delta_{\infty}u_{\varepsilon,j}\right)+\delta\Delta u_{\varepsilon,j}\right)\\
 & \leq\frac{\delta^{\frac{p_{j}(x)-2}{2}}+\left(\delta+M^{2}\right)^{\frac{p_{j}(x)-2}{2}}}{\delta+\left|Du_{\varepsilon,j}\right|^{2}}C(q,\varepsilon,u)\left(\left|Du_{\varepsilon,j}\right|^{2}\left(N+p_{j}(x)-2\right)+\delta N\right)\\
 & \leq C(q,\varepsilon,u)\left(\delta^{\frac{p^{-}-2}{2}}+\left(\delta+M^{2}\right)^{\frac{p^{+}-2}{2}}\right)\left(2N+p^{+}-2\right),
\end{align*}
where the first inequality follows like estimate (\ref{eq:ishiiq_main})
since $p_{j}\geq p^{-}>1$.

\textbf{(Step 2) }We let $\delta\rightarrow0$ in the auxiliary inequality
(\ref{eq:aux}). The RHS becomes
\[
\int_{\Omega_{r(\varepsilon)}\setminus\left\{ Du_{\varepsilon}=0\right\} }\left|Du_{\varepsilon}\right|^{p(x)-2}Du_{\varepsilon}\cdot\left(D\varphi+\log\left|Du_{\varepsilon}\right|Dp\,\varphi\right)\d x
\]
by the Lebesgue's dominated convergence theorem. We intend to apply
Fatou's lemma on the LHS. We have $\left(Du_{\varepsilon}(x),D^{2}u_{\varepsilon}(x)\right)\in J^{2,-}u_{\varepsilon}(x)$
for almost every $x\in\Omega_{r(\varepsilon)}$. Therefore by Lemma
\ref{lem:ishiiq} it holds that
\begin{equation}
\left|Du_{\varepsilon}\right|^{\min(p(x)-2,0)}F(x,Du_{\varepsilon},D^{2}u_{\varepsilon})\geq E(\varepsilon)\text{ in }\left\{ x\in\Omega_{r(\varepsilon)}:Du_{\varepsilon}\not=0\right\} \label{eq:ishiiq_ae}
\end{equation}
and by the property (v) in Lemma \ref{lem:infconv_properties} we
have 
\begin{equation}
D^{2}u_{\varepsilon}\leq\frac{q-1}{\varepsilon}\left|Du_{\varepsilon}\right|^{\frac{q-2}{q-1}}I.\label{eq:ishiiq_hessianineq}
\end{equation}
Observe that since $q>2$, the condition (\ref{eq:ishiiq_hessianineq})
implies that the Hessian $D^{2}u_{\varepsilon}$ is negative semi-definite
in the set where the gradient $Du_{\varepsilon}$ vanishes. Using
this fact, Fatou's lemma and (\ref{eq:ishiiq_ae}) we get 
\begin{align}
\liminf_{\delta\rightarrow0} & \int_{\Omega_{r(\varepsilon)}}-\left(\left|Du_{\varepsilon}\right|^{2}+\delta\right)^{\frac{p(x)-2}{2}}\left(\Delta u_{\varepsilon}+\frac{p(x)-2}{\left|Du_{\varepsilon}\right|^{2}+\delta}\Delta_{\infty}u_{\varepsilon}\right)\varphi\d x\nonumber \\
\geq & \liminf_{\delta\rightarrow0}\int_{\left\{ Du_{\varepsilon}\not=0\right\} }-\left(\left|Du_{\varepsilon}\right|^{2}+\delta\right)^{\frac{p(x)-2}{2}}\left(\Delta u_{\varepsilon}+\frac{p(x)-2}{\left|Du_{\varepsilon}\right|^{2}+\delta}\Delta_{\infty}u_{\varepsilon}\right)\varphi\d x\nonumber \\
 & +\liminf_{\delta\rightarrow0}\int_{\left\{ Du_{\varepsilon}=0\right\} }-\delta^{\frac{p(x)-2}{2}}\Delta u_{\varepsilon}\varphi\d x\nonumber \\
\geq & \int_{\left\{ Du_{\varepsilon}\not=0\right\} }-\left|Du_{\varepsilon}\right|^{p(x)-2}\left(\Delta u_{\varepsilon}+\frac{p(x)-2}{\left|Du_{\varepsilon}\right|^{2}}\Delta_{\infty}u_{\varepsilon}\right)\varphi\d x\nonumber \\
\geq & E(\varepsilon)\int_{\left\{ Du_{\varepsilon}\not=0\right\} }\left|Du_{\varepsilon}\right|^{\max(p(x)-2,0)}\varphi\d x,\label{eq:proof_lemma_liminfstuff}
\end{align}
and thus we arrive at the desired inequality. Our use of Fatou's lemma
is justified since if $Du_{\varepsilon}\not=0$ and $p(x)\leq2$,
we have by (\ref{eq:ishiiq_hessianineq})
\begin{align*}
\Big(\left|Du_{\varepsilon}\right|^{2} & +\delta\Big)^{\frac{p(x)-2}{2}}\left(\Delta u_{\varepsilon}+\frac{p(x)-2}{\left|Du_{\varepsilon}\right|^{2}+\delta}\Delta_{\infty}u_{\varepsilon}\right)\\
= & \frac{\left(\left|Du_{\varepsilon}\right|^{2}+\delta\right)}{\left|Du_{\varepsilon}\right|^{2}+\delta}^{\frac{p(x)-2}{2}}\left(\left|Du_{\varepsilon}\right|^{2}\left(\Delta u_{\varepsilon}+\frac{p(x)-2}{\left|Du_{\varepsilon}\right|^{2}}\Delta_{\infty}u_{\varepsilon}\right)+\delta\Delta u_{\varepsilon}\right)\\
\leq & \frac{\left(\left|Du_{\varepsilon}\right|^{2}+\delta\right)}{\left|Du_{\varepsilon}\right|^{2}+\delta}^{\frac{p(x)-2}{2}}\frac{q-1}{\varepsilon}\left(\left|Du_{\varepsilon}\right|^{\frac{q-2}{q-1}+2}\left(N+p(x)-2\right)+\left|Du_{\varepsilon}\right|^{\frac{q-2}{q-1}}\delta N\right)\\
\leq & \left|Du_{\varepsilon}\right|^{p(x)-2+\frac{q-2}{q-1}}\left(\frac{q-1}{\varepsilon}\right)\left(2N+p(x)-2\right)\\
\leq & \left(\left\Vert Du_{\varepsilon}\right\Vert _{L^{\infty}(\supp\varphi)}+1\right)^{p^{+}-2+\frac{q-2}{q-1}}\left(\frac{q-1}{\varepsilon}\right)\left(2N+p^{+}-2\right),
\end{align*}
where the last inequality follows from $p^{-}-2+\frac{q-2}{q-1}\geq0$.
If $Du_{\varepsilon}\not=0$ and $p(x)>2$, we have simply
\begin{align*}
\Big(\left|Du_{\varepsilon}\right|^{2} & +\delta\Big)^{\frac{p(x)-2}{2}}\left(\Delta u_{\varepsilon}+\frac{p(x)-2}{\left|Du_{\varepsilon}\right|^{2}+\delta}\Delta_{\infty}u_{\varepsilon}\right)\\
\leq & \left(\left\Vert Du_{\varepsilon}\right\Vert _{L^{\infty}(\supp\varphi)}^{2}+1\right)^{\frac{p^{+}-2}{2}+\frac{q-2}{q-1}}(\frac{q-1}{\varepsilon})\left(N+p^{+}-2\right).\qedhere
\end{align*}
\end{proof}
In the next two lemmas we use Caccioppoli type estimates and algebraic
inequalities to show that the sequence of inf-convolutions converges
to the viscosity supersolution in $W_{loc}^{1,p(\cdot)}(\Omega)$.
\begin{lem}
\label{lem:proof_lemma2}Under the assumptions of Lemma \ref{lem:proof_lemma},
the function $u$ belongs in $W_{loc}^{1,p(\cdot)}(\Omega)$ and for
any $\Omega^{\prime}\Subset\Omega$ we have $Du_{\varepsilon}\rightarrow Du$
weakly in $L^{p(\cdot)}(\Omega^{\prime})$ for some subsequence.
\end{lem}
\begin{proof}
Take a cut-off function $\xi\in C_{0}^{\infty}(\Omega^{\prime})$
such that $0\leq\xi\leq1$ in $\Omega$ and $\xi\equiv1$ in $\Omega^{\prime}$.
Then assume that $\varepsilon$ is so small that $\mathrm{supp}\thinspace\xi=:K\subset\Omega{}_{r(\varepsilon)}$.
We define a test function $\varphi:=(L-u_{\varepsilon})\xi^{p^{+}}$
where $L:=\sup_{\varepsilon,x\in\Omega^{\prime}}\left|u_{\varepsilon}(x)\right|$
is finite since $u_{\varepsilon}\rightarrow u$ locally uniformly.
We have
\[
D\varphi=-Du_{\varepsilon}\,\xi^{p^{+}}+(L-u_{\varepsilon})p^{\text{+}}\xi^{p^{+}-1}D\xi
\]
and therefore by Lemma \ref{lem:proof_lemma}
\begin{align*}
\int_{\Omega_{r(\varepsilon)}}\left|Du_{\varepsilon}\right|^{p(x)}\xi^{p^{+}}\d x\leq & \int_{\Omega_{r(\varepsilon)}}\left|Du_{\varepsilon}\right|^{p(x)-1}\xi^{p^{+}-1}\left(L-u_{\varepsilon}\right)p^{+}\left|D\xi\right|\d x\\
 & +\int_{\Omega_{r(\varepsilon)}}\left|Du_{\varepsilon}\right|^{p(x)-1}\left|\log\left|Du_{\varepsilon}\right|\right|\left|Dp\right|\left(L-u_{\varepsilon}\right)\xi^{p^{+}}\d x\\
 & +\left|E(\varepsilon)\right|\int_{\Omega_{r(\varepsilon)}}\left|Du_{\varepsilon}\right|^{\max(p(x)-2,0)}\left(L-u_{\varepsilon}\right)\xi^{p^{+}}\d x\\
=: & I_{1}+I_{2}+I_{3}.
\end{align*}
We estimate these integrals using Young's inequality. The first integral
is estimated by the facts $\frac{p(x)\left(p^{+}-1\right)}{p(x)-1}\geq p^{+}$
and $\xi\leq1$ as follows 
\begin{align*}
I_{1}\leq & \int_{\Omega_{r(\varepsilon)}}\delta\left|Du_{\varepsilon}\right|^{p(x)}\xi^{\frac{p(x)(p^{+}-1)}{p(x)-1}}+\left(\frac{2}{\delta}Lp^{+}\left|D\xi\right|\right)^{p(x)}\d x\\
\leq & \delta\int_{\Omega_{r(\varepsilon)}}\left|Du_{\varepsilon}\right|^{p(x)}\xi^{p^{+}}\d x+C(\delta,p,L,D\xi).
\end{align*}
To estimate $I_{2}$, we also use the inequality $a^{s}\left|\log a\right|\leq a^{s+\frac{1}{2}}+\frac{1}{s}\text{ for }a>0\text{ and }s>0$,
\begin{align*}
I_{2}\leq & \int_{\Omega_{r(\varepsilon)}}\left(\left|Du_{\varepsilon}\right|^{p(x)-\frac{1}{2}}+\frac{1}{p(x)-1}\right)\xi^{p^{+}}\left|Dp\right|2L\d x\\
\leq & \int_{\Omega_{r(\varepsilon)}}\delta\left|Du_{\varepsilon}\right|^{p(x)}\xi^{\frac{p^{+}p(x)}{p(x)-\frac{1}{2}}}+\left(\frac{2}{\delta}\left|Dp\right|L\right)^{2p(x)}+\frac{2L\left|Dp\right|\xi^{p^{+}}}{p^{-}-1}\d x\\
\leq & \delta\int_{\Omega_{r(\varepsilon)}}\left|Du_{\varepsilon}\right|^{p(x)}\xi^{p^{+}}\d x+C(\delta,p,Dp,L).
\end{align*}
The last integral is estimated by the two alternatives in $\max(p(x)-2,0)$
as follows (we may assume that $\left|E(\varepsilon)\right|\leq1$)
\begin{align*}
I_{3}\leq & \int_{\Omega_{r(\varepsilon)}\cap\left\{ p(x)>2\right\} }\left|Du_{\varepsilon}\right|^{p(x)-2}\xi^{p^{+}}2L\d x+\int_{\Omega_{r(\varepsilon)}\cap\left\{ p(x)\leq2\right\} }2L\xi^{p^{+}}\d x\\
\leq & \int_{\Omega_{r(\varepsilon)}\cap\left\{ p(x)>2\right\} }\delta\left|Du_{\varepsilon}\right|^{p(x)}\xi^{\frac{p^{+}p(x)}{p(x)-2}}+\left(\frac{2}{\delta}L\right)^{\frac{p(x)}{2}}\d x+C(p,L)\\
\leq & \delta\int_{\Omega_{r(\varepsilon)}}\left|Du_{\varepsilon}\right|^{p(x)}\xi^{p^{+}}\d x+C(\delta,p,L).
\end{align*}
Taking small $\delta$ we conclude that $Du_{\varepsilon}$ is bounded
in $L^{p(\cdot)}(\Omega^{\prime})$ with respect to $\varepsilon$.
Since $L^{p(\cdot)}(\Omega^{\prime})$ is a reflexive Banach space
\cite[p76 and p89]{pxbook}, it follows that there is a function $Du\in L^{p(\cdot)}(\Omega^{\prime})$
such that $Du_{\varepsilon}\rightarrow Du$ weakly in $L^{p(\cdot)}(\Omega^{\prime})$
for some subsequence. Consequently $u\in W^{1,p(\cdot)}(\Omega^{\prime})$
with $Du$ as its weak derivative.
\end{proof}
\begin{lem}
\label{lem:proof_lemma_convergence}Under the assumptions of Lemma
\ref{lem:proof_lemma}, for any $\Omega^{\prime}\Subset\Omega$ we
have $Du_{\varepsilon}\rightarrow Du$ in $L^{p(\cdot)}(\Omega^{\prime})$
for some subsequence.
\end{lem}
\begin{proof}
Take a cut-off function $\xi\in C_{0}^{\infty}(\Omega)$ such that
$\xi\equiv1$ in $\Omega^{\prime}$ and define a test function $\varphi:=(u-u_{\varepsilon})\xi$.
Then assume that $\varepsilon$ is so small that $\supp\xi=:K\subset\Omega_{r(\varepsilon)}$.
Since $\varphi\in W^{1,p(\cdot)}(\Omega_{r(\varepsilon)})$ with compact
support it follows from Lemma \ref{lem:proof_lemma} that
\begin{align}
\int_{\Omega_{r(\varepsilon)}} & \left(\left|Du\right|^{p(x)-2}Du-\left|Du_{\varepsilon}\right|^{p(x)-2}Du_{\varepsilon}\right)\cdot\left(Du-Du_{\varepsilon}\right)\xi\d x\nonumber \\
\leq & \int_{\Omega_{r(\varepsilon)}}\left|Du_{\varepsilon}\right|^{p(x)-2}Du_{\varepsilon}\cdot D\xi\,(u-u_{\varepsilon})\d x\nonumber \\
 & +\int_{\Omega_{r(\varepsilon)}}\left|Du_{\varepsilon}\right|^{p(x)-2}\log\left(\left|Du_{\varepsilon}\right|\right)Du_{\varepsilon}\cdot Dp\,(u-u_{\varepsilon})\xi\d x\nonumber \\
 & +\left|E(\varepsilon)\right|\int_{\Omega_{r(\varepsilon)}}\left|Du_{\varepsilon}\right|^{\max(p(x)-2,0)}(u-u_{\varepsilon})\xi\d x\nonumber \\
 & +\int_{\Omega_{r(\varepsilon)}}\left|Du\right|^{p(x)-2}Du\cdot\left(Du-Du_{\varepsilon}\right)\xi\d x\nonumber \\
\leq & \left\Vert u-u_{\varepsilon}\right\Vert _{L^{\infty}(K)}\int_{K}\left(C(p^{-})+\left|Du_{\varepsilon}\right|^{p(x)}\right)\left(D\xi+\left|Dp\right|+\left|E(\varepsilon)\right|\right)\d x\nonumber \\
 & +\int_{K}\left|Du\right|^{p(x)-2}Du\cdot\left(Du-Du_{\varepsilon}\right)\xi\d x.\label{eq:proof_2}
\end{align}
According to Lemma \ref{lem:proof_lemma2} we have $u_{\varepsilon}\rightarrow u$
locally uniformly and $Du_{\varepsilon}\rightarrow Du$ weakly in
$L^{p(\cdot)}(K)$ for a subsequence. Thus by passing to a subsequence
we may assume that the right hand side of (\ref{eq:proof_2}) converges
to zero. The claim now follows from the inequalities (see e.g. \cite[Chapter 12]{lindqvist_plaplace})
\begin{align*}
\big(\left|a\right|^{p(x)-2}a & -\left|b\right|^{p(x)-2}b\big)\cdot\left(a-b\right)\\
 & \geq\begin{cases}
(p(x)-1)\left|a-b\right|^{2}\left(1+\left|a\right|^{2}+\left|b\right|^{2}\right)^{\frac{p(x)-2}{2}} & p(x)<2\\
2^{2-p(x)}\left|a-b\right|^{p(x)} & p(x)\geq2
\end{cases}
\end{align*}
for $a,b\in\mathbb{R}^{N}$. Indeed, we immediately get that $\int_{\Omega^{\prime}\cap\left\{ p(x)\geq2\right\} }\left|Du-Du_{\varepsilon}\right|^{p(x)}\d x\rightarrow0$.
To deal with the set $\left\{ p(x)<2\right\} $, we first apply the
above algebraic inequality and then estimate using Hölder's inequality,
the modular inequality (\ref{eq:modular_ineq}) and the definition
of the $\left\Vert \cdot\right\Vert _{L^{p(\cdot)}}$-norm. We get
\begin{align*}
 & \int_{\Omega^{\prime}\cap\left\{ p(x)<2\right\} }\left|Du-Du_{\varepsilon}\right|^{p(x)}\d x\\
 & \ \leq\int_{\Omega^{\prime}\cap\left\{ p(x)<2\right\} }\left(\left(\left|Du\right|^{p(x)-2}Du-\left|Du_{\varepsilon}\right|^{p(x)-2}Du_{\varepsilon}\right)\cdot\left(Du-Du_{\varepsilon}\right)\right)^{\frac{p(x)}{2}}\\
 & \ \ \ \ \ \cdot\left(\frac{1}{p(x)-1}\right)^{\frac{p(x)}{2}}\left(1+\left|Du\right|^{2}+\left|Du_{\varepsilon}\right|^{2}\right)^{\frac{p(x)\left(2-p(x)\right)}{4}}\d x\\
 & \ \leq\left\Vert \left(\left(\left|Du\right|^{p(x)-2}Du-\left|Du_{\varepsilon}\right|^{p(x)-2}Du_{\varepsilon}\right)\cdot\left(Du-Du_{\varepsilon}\right)\right)^{\frac{p(x)}{2}}\right\Vert _{L^{\frac{2}{p(\cdot)}}(\Omega^{\prime}\cap\left\{ p(x)<2\right\} )}\\
 & \ \ \ \ \ \cdot\frac{2}{p^{-}-1}\left\Vert \left(1+\left|Du\right|^{2}+\left|Du_{\varepsilon}\right|^{2}\right)^{\frac{p(x)\left(2-p(x)\right)}{4}}\right\Vert _{L^{\frac{2}{2-p(\cdot)}}(\Omega^{\prime}\cap\left\{ p(x)<2\right\} )}\\
 & \ \leq\left(\int_{\Omega_{r(\varepsilon)}}\left(\left|Du\right|^{p(x)-2}Du-\left|Du_{\varepsilon}\right|^{p(x)-2}Du_{\varepsilon}\right)\cdot\left(Du-Du_{\varepsilon}\right)\xi\d x\right)^{s}\\
 & \ \ \ \ \ \cdot\frac{2}{p^{-}-1}\left(1+\int_{\Omega^{\prime}\cap\left\{ p(x)<2\right\} }\left(1+\left|Du\right|^{2}+\left|Du_{\varepsilon}\right|^{2}\right)^{\frac{p(x)}{2}}\d x\right),
\end{align*}
where $s\in\left\{ \frac{p^{+}}{2},\frac{p^{-}}{2}\right\} $. The
last integral is bounded since the sequence $Du_{\varepsilon}$ is
bounded in $L^{p(\cdot)}(\Omega^{\prime})$ by its weak convergence.
The RHS therefore converges to zero by (\ref{eq:proof_2}).
\end{proof}
Next, we use the previous convergence result to pass to the limit
in the inequality of Lemma \ref{lem:proof_lemma} and conclude that
viscosity supersolutions to $-\Delta_{p(x)}^{N}u\geq0$ are weak supersolutions
to $-\Delta_{p(x)}^{S}u\geq0$. 
\begin{thm}
\label{thm:sobolev} If $u\in C(\Omega)$ is a viscosity supersolution
to $-\Delta_{p(x)}^{N}u\geq0$ in $\Omega$, then $u$ is a weak supersolution
to $-\Delta_{p(x)}^{S}u\geq0$ in $\Omega$.
\end{thm}
\begin{proof}
It is clear from the definition of weak supersolutions to $-\Delta_{p(x)}^{S}u\geq0$
that we can without loss of generality assume that $u$ is uniformly
continuous in $\Omega$ by restricting to a smaller domain. Fix a
non-negative test function $\varphi\in C_{0}^{\infty}(\Omega)$ and
take open $\Omega^{\prime}\Subset\Omega$ such that $\mathrm{supp}\thinspace\varphi\subset\Omega^{\prime}$.
Let $q$ and $u_{\varepsilon}$ be as in Lemma \ref{lem:proof_lemma}
and assume that $\varepsilon$ is so small that $\Omega^{\prime}\subset\Omega_{r(\varepsilon)}$.
Then the claim follows from Lemma \ref{lem:proof_lemma} if we show
that
\begin{equation}
\lim_{\varepsilon\rightarrow0}\int_{\Omega^{\prime}}\left|Du_{\varepsilon}\right|^{p(x)-2}Du_{\varepsilon}\cdot D\varphi\d x=\int_{\Omega^{\prime}}\left|Du\right|^{p(x)-2}Du\cdot D\varphi\d x\label{eq:ptermConv}
\end{equation}
and
\begin{align}
\lim_{\varepsilon\rightarrow0} & \int_{\Omega^{\prime}}\left|Du_{\varepsilon}\right|^{p(x)-2}\log\left(\left|Du_{\varepsilon}\right|\right)Du_{\varepsilon}\cdot Dp\,\varphi\d x\nonumber \\
= & \int_{\Omega^{\prime}}\left|Du\right|^{p(x)-2}\log\left(\left|Du\right|\right)Du\cdot Dp\,\varphi\d x\label{eq:logtermConv}
\end{align}
as well as
\begin{equation}
\lim_{\varepsilon\rightarrow0}E(\varepsilon)\int_{\Omega^{\prime}}\left|Du_{\varepsilon}\right|^{\max(p(x)-2,0)}\varphi\d x=0.\label{eq:errortermConv}
\end{equation}
By Lemma \ref{lem:proof_lemma_convergence} we have that $u_{\varepsilon}\rightarrow u$
in $W^{1,p(\cdot)}(\Omega^{\prime})$.\textbf{}\\
\textbf{Claim} (\ref{eq:ptermConv}) follows from the inequalities
(see e.g. \cite[Chapter 12]{lindqvist_plaplace})
\begin{equation}
\left|\left|a\right|^{p(x)-2}a-\left|b\right|^{p(x)-2}b\right|\leq\begin{cases}
2^{2-p(x)}\left|a-b\right|^{p(x)-1} & p(x)<2\\
2^{-1}\left(\left|a\right|^{p(x)-2}+\left|b\right|^{p(x)-2}\right)\left|a-b\right| & p(x)\geq2
\end{cases}\label{eq:proof_elemineq}
\end{equation}
for $a,b\in\mathbb{R}^{N}$. Indeed, when $\varepsilon$ is so small
that $\int_{\Omega^{\prime}}\left|Du_{\varepsilon}-Du\right|^{p(x)}\d x<1$
we have by Hölder's inequality and the modular inequality
\begin{align*}
\int_{\Omega^{\prime}} & \left|\left|Du_{\varepsilon}\right|^{p(x)-2}Du_{\varepsilon}-\left|Du\right|^{p(x)-2}Du\right|\d x\\
\leq & 2\int_{\Omega^{\prime}\cap\left\{ p(x)<2\right\} }\left|Du_{\varepsilon}-Du\right|^{p(x)-1}\d x\\
 & +2^{-1}\int_{\Omega^{\prime}\cap\left\{ p(x)\geq2\right\} }\left(\left|Du_{\varepsilon}\right|^{p(x)-2}+\left|Du\right|^{p(x)-2}\right)\left|Du_{\varepsilon}-Du\right|\d x\\
\leq & C(p,\Omega)\left(\int_{\Omega^{\prime}}\left|Du_{\varepsilon}-Du\right|^{p(x)}\d x\right)^{\frac{1}{p^{+}}}\\
 & +C(p,\Omega)\left(1+\int_{\Omega^{\prime}}\left|Du_{\varepsilon}\right|^{p(x)}+\left|Du\right|^{p(x)}\d x\right)\left\Vert Du_{\varepsilon}-Du\right\Vert _{L^{p(\cdot)}(\Omega^{\prime})}.
\end{align*}
\textbf{Claim} (\ref{eq:errortermConv}) holds since $\int_{\Omega^{\prime}}\left|Du_{\varepsilon}\right|^{p(x)}\d x$
is bounded and $E(\varepsilon)\rightarrow0$.\\
\textbf{Claim} (\ref{eq:logtermConv}) follows if we show that
\begin{equation}
\lim_{\varepsilon\rightarrow0}\int_{\Omega^{\prime}}\left|\left|Du_{\varepsilon}\right|^{p(x)-2}\log\left(\left|Du_{\varepsilon}\right|\right)Du_{\varepsilon}-\left|Du\right|^{p(x)-2}\log\left(\left|Du\right|\right)Du\right|\d x=0.\label{eq:proof_Gconv}
\end{equation}
 To this end, fix $0<\epsilon<1$. The mapping $(a,x)\mapsto\left|a\right|^{p(x)-2}\log\left(\left|a\right|\right)a$
is uniformly continuous in bounded sets of $\mathbb{R}^{N}\times\Omega^{\prime}$.
Hence there exists $\delta=\delta(\epsilon)<\epsilon$ such that whenever
$x\in\Omega^{\prime}$ and $a,b\in\overline{B}(0,3)$ satisfy $\left|a-b\right|<\delta$,
it holds
\begin{equation}
\left|\left|a\right|^{p(x)-2}\log\left(\left|a\right|\right)a-\left|b\right|^{p(x)-2}\log\left(\left|b\right|\right)b\right|\leq\epsilon.\label{eq:proof_uniform}
\end{equation}
If $\left|a\right|,\left|b\right|\geq1$ and $\left|a-b\right|<\delta$,
then we use (\ref{eq:proof_elemineq}) to get the estimate
\begin{align}
\Big|\left|a\right| & ^{p(x)-2}\log\left(\left|a\right|\right)a-\left|b\right|^{p(x)-2}\log\left(\left|b\right|\right)b\Big|\nonumber \\
\leq & \left|b\right|^{p(x)-1}\left|\log\left|a\right|-\log\left|b\right|\right|+\left|\log\left|a\right|\right|\left|\left|a\right|^{p(x)-2}a-\left|b\right|^{p(x)-2}b\right|\nonumber \\
\leq & \left|b\right|^{p(x)}\left|a-b\right|+\left|a\right|\cdot\begin{cases}
2^{2-p(x)}\left|a-b\right|^{p(x)-1}, & p(x)<2\\
2^{-1}\left(\left|a\right|^{p(x)-2}+\left|b\right|^{p(x)-2}\right)\left|a-b\right|, & p(x)\geq2
\end{cases}\nonumber \\
\leq & (1+2^{-1})\left(\left|a\right|^{p(x)}+\left|b\right|^{p(x)}\right)\left|a-b\right|+2\left|a\right|\left|a-b\right|^{p(x)-1}\nonumber \\
\leq & C\left(\left|a\right|^{p(x)}+\left|b\right|^{p(x)}\right)\epsilon^{\min(p^{-}-1,1)}.\label{eq:proof_elem2}
\end{align}
We denote 
\begin{align*}
F_{\varepsilon}= & \left\{ x\in\Omega^{\prime}:\left|Du_{\varepsilon}(x)-Du(x)\right|\geq\delta\right\} .
\end{align*}
The strong convergence of $Du_{\varepsilon}$ to $Du$ in $L^{p(\cdot)}(\Omega^{\prime})$
implies that $Du_{\varepsilon}\rightarrow Du$ in measure in $\Omega^{\prime}$
(see \cite[Lemma 3.2.10]{pxbook}). Thus there is $\varepsilon_{0}=\varepsilon_{0}(\delta)$
such that for all $\varepsilon<\varepsilon_{0}$ it holds $\left|F_{\varepsilon}\right|\leq\delta$.
Using the inequality $a^{s}\left|\log a\right|\leq a^{s+\frac{1}{2}}+\frac{1}{s}\text{ for }a>0\text{ and }s>0$
we get for all $\varepsilon<\varepsilon_{0}$
\begin{align}
\int_{F_{\varepsilon}} & \left|\left|Du_{\varepsilon}\right|^{p(x)-2}\log\left(\left|Du_{\varepsilon}\right|\right)Du_{\varepsilon}-\left|Du\right|^{p(x)-2}\log\left(\left|Du\right|\right)Du\right|\d x\nonumber \\
\leq & \int_{F_{\varepsilon}}\frac{2}{p(x)-1}+\left|Du_{\varepsilon}\right|^{p(x)-\frac{1}{2}}+\left|Du\right|^{p(x)-\frac{1}{2}}\d x\nonumber \\
\leq & C(p^{-})\left|F_{\varepsilon}\right|+\left\Vert 1\right\Vert _{L^{2p(\cdot)}(F_{\varepsilon})}\left(\left\Vert Du_{\varepsilon}\right\Vert _{L^{\frac{p(\cdot)}{p(\cdot)-\frac{1}{2}}}(F_{\varepsilon})}+\left\Vert Du\right\Vert _{L^{\frac{p(\cdot)}{p(\cdot)-\frac{1}{2}}}(F_{\varepsilon})}\right)\nonumber \\
\leq & C(p^{-})\left|F_{\varepsilon}\right|+\left|F_{\varepsilon}\right|^{\frac{1}{2p^{+}}}\left(2+\int_{F_{\varepsilon}}\left|Du_{\varepsilon}\right|^{p(x)}+\left|Du\right|^{p(x)}\d x\right)\nonumber \\
\leq & C(p^{-})\left(1+\int_{\Omega^{\prime}}\left|Du_{\varepsilon}\right|^{p(x)}+\left|Du\right|^{p(x)}\d x\right)\epsilon^{\frac{1}{2p^{+}}}.\label{eq:proof_F}
\end{align}
If $x\in\Omega^{\prime}\setminus F_{\varepsilon}$, then either $\left|Du_{\varepsilon}\right|,\left|Du\right|\leq3$
or $\left|Du_{\varepsilon}\right|,\left|Du\right|\geq1$. Hence by
(\ref{eq:proof_uniform}) and (\ref{eq:proof_elem2}) we have 
\begin{align}
\int_{\Omega^{\prime}\setminus F_{\varepsilon}} & \left|\left|Du_{\varepsilon}\right|^{p(x)-2}\log\left(\left|Du_{\varepsilon}\right|\right)Du_{\varepsilon}-\left|Du\right|^{p(x)-2}\log\left(\left|Du\right|\right)Du\right|\d x\nonumber \\
\leq & C\left(\int_{\Omega^{\prime}}\left|Du_{\varepsilon}\right|^{p(x)}+\left|Du\right|^{p(x)}+1\d x\right)\epsilon^{\min(p^{-}-1,1)}.\label{eq:proofAB}
\end{align}
 Combining (\ref{eq:proofAB}) and (\ref{eq:proof_F}) proves (\ref{eq:proof_Gconv})
since $\epsilon$ was arbitrary.
\end{proof}
Merging Theorems \ref{thm:visc} and \ref{thm:sobolev} yields the
following equivalence result.
\begin{thm}
\label{thm:equivalence}A function $u$ is a viscosity solution to
$-\Delta_{p(x)}^{N}u=0$ in $\Omega$ if and only if it is a weak
solution to $-\Delta_{p(x)}^{S}u=0$ in $\Omega$.
\end{thm}
Since the weak solutions to the strong $p(x)$-Laplace equation are
locally $C^{1,\alpha}$ continuous \cite{strongpx_regularity}, our
equivalence result yields local $C^{1,\alpha}$ regularity also for
viscosity solutions of the normalized $p(x)$-Laplace equation.
\begin{cor}
If $u$ is a viscosity solution to $-\Delta_{p(x)}^{N}u=0$ in a bounded
domain $\Omega$, then $u\in C^{1,\alpha}(\Omega)$ with $\alpha\in(0,1)$.
\end{cor}

\section{An Application: A Radó-type removability theorem}

The classical theorem of Radó says that if a continuous complex-valued
function $f$ defined on a domain $\Omega\subset\mathbb{C}$ is holomorphic
in $\Omega\setminus\left\{ f=0\right\} $, then it is holomorphic
in the whole $\Omega$. Similar results have been proven for solutions
of partial differential equations. We prove a Radó-type removability
theorem for the strong $p(x)$-Laplace equation. It is worth pointing
out that it could be difficult to show this kind of result without
appealing to viscosity solutions whereas it is straightforward to
do so with the help of the equivalence result. The theorem follows
by observing that weak solutions to $\Delta_{p(x)}^{S}u=0$ coincide
with viscosity solutions of an equation that satisfies the assumptions
of a Radó-type removability theorem in \cite{quasiremove}.

Recall that we ignore the test functions whose gradient vanishes at
the point of touching in the Definition \ref{def:viscositysuper}
of viscosity solutions to $-\Delta_{p(x)}^{N}u=0$. Sometimes this
kind of solutions are called \textit{feeble viscosity solutions} (e.g.
\cite{quasiremove,convex_functionals}). We will observe that these
feeble viscosity solutions to $-\Delta_{p(x)}^{N}u=0$ are exactly
the usual viscosity solutions to
\begin{equation}
-\tr(A(x,Du)D^{2}u)=0,\label{eq:ordinary_eq}
\end{equation}
where $A(x,Du):=\left|Du\right|^{2}I+\left(p(x)-2\right)Du\otimes Du$.
To be precise, we define the viscosity solutions to (\ref{eq:ordinary_eq}).
\begin{defn}
A lower semicontinuous function $u$ is a \textit{viscosity supersolution}
to (\ref{eq:ordinary_eq}) in $\Omega$ if, whenever $(\eta,X)\in J^{2,-}u(x)$
with $x\in\Omega$, then 
\[
-\tr(A(x,\eta)X)\geq0.
\]
A function $u$ is a \textit{viscosity subsolution} to (\ref{eq:ordinary_eq})
if $-u$ is a supersolution, and a \textit{viscosity solution} if
it is both viscosity super- and subsolution.
\end{defn}
\begin{lem}
\label{lem:ordinary_lemma}A function $u$ is a viscosity solution
to $-\Delta_{p(x)}^{N}u=0$ if and only if it is a viscosity solution
to (\ref{eq:ordinary_eq}).
\end{lem}
\begin{proof}
It is enough to consider supersolutions. Take $(\eta,X)\in J^{2,-}u(x)$
with $x\in\Omega$. If $\eta=0$, then the conditions for both definitions
are satisfied, so we may assume that $\eta\not=0$. Then we have
\[
F(x,\eta,X)\geq0
\]
 if and only if
\[
-\big(\left|\eta\right|^{2}\tr(X)+\left(p(x)-2\right)\left\langle X\eta,\eta\right\rangle \big)\geq0,
\]
where 
\begin{align*}
\left|\eta\right|^{2}\tr(X)+\left(p(x)-2\right)\left\langle X\eta,\eta\right\rangle = & \left|\eta\right|^{2}\tr(X)+\left(p(x)-2\right)\tr(\eta\otimes\eta X)\\
= & \tr\big(\big(\left|\eta\right|^{2}I+\left(p(x)-2\right)\eta\otimes\eta\big)X\big).
\end{align*}
Hence the definitions are equivalent.
\end{proof}
\begin{thm}[A Radó-type removability theorem]
 Let $u\in C^{1}(\Omega)$ be a weak solution to $-\Delta_{p(x)}^{S}u=0$
in $\Omega\setminus\left\{ u=0\right\} $. Then $u$ is a weak solution
to $-\Delta_{p(x)}^{S}u=0$ in the whole $\Omega$.
\end{thm}
\begin{proof}
By Lemma \ref{lem:ordinary_lemma} and our equivalence result weak
solutions to $-\Delta_{p(x)}^{S}u=0$ coincide with viscosity solutions
to (\ref{eq:ordinary_eq}). Therefore it suffices to show that if
$u$ is a viscosity solution to (\ref{eq:ordinary_eq}) in $\Omega\setminus\left\{ u=0\right\} $,
it is a viscosity solution to (\ref{eq:ordinary_eq}) in the whole
$\Omega$. This on the other hand follows from \cite[Theorem 2.2]{quasiremove}.
The matrix $A$ satisfies the assumptions of the theorem as it is
symmetric, has continuous entries and $A(x,0,0)=0$ for all $x\in\Omega$.
It is also positive semi-definite since for all $\xi\in\mathbb{R}^{N}$
we have
\begin{align*}
\xi^{\prime}\left(\left|\eta\right|^{2}I+\left(p(x)-2\right)\eta\otimes\eta\right)\xi\geq & \xi^{\prime}\left(\left|\eta\right|^{2}I-\eta\otimes\eta\right)\xi\\
\geq & \left|\xi\right|^{2}\left(\left|\eta\right|^{2}-\left\Vert \eta\otimes\eta\right\Vert \right)=0.\qedhere
\end{align*}
\end{proof}
\bibliographystyle{alpha}
\bibliography{equivalence}

\begin{thebibliography}{DHHR11}

\bibitem[AH10]{intr_strongpx_1}
T.~Adamowicz and P.~H{\"a}st{\"o}.
\newblock Mappings of finite distortion and {PDE} with nonstandard growth.
\newblock {\em Int. Math. Res. Not. IMRN}, pages 1940--1965, 2010.

\bibitem[AH11]{intr_strongpx_2}
T.~Adamowicz and P.~H{\"a}st{\"o}.
\newblock Harnack's inequality and the strong $p(\cdot)$-{L}aplacian.
\newblock {\em J. Differential Equations}, 250:1631--1649, 2011.

\bibitem[AHP17]{gamestuff}
{\'A}.~Arroyo, J.~Heino, and M.~Parviainen.
\newblock Tug-of-war games with varying probabilities and the normalized
  $p(x)$-{L}aplacian.
\newblock {\em {C}ommun. {P}ure {A}ppl. {A}nal.}, 16(3):915--944, 2017.

\bibitem[APR17]{OptimalC1}
A.~Attouchi, M.~Parviainen, and E.~Ruosteenoja.
\newblock ${C}^{1,\alpha}$ regularity for the normalized $p$-{P}oisson problem.
\newblock {\em {J}. {M}ath. {P}ures {A}ppl.}, 108(4):553--591, 2017.

\bibitem[BG15]{banerjeeGarofalo15}
A.~Banerjee and N.~Garofalo.
\newblock Modica type gradient estimates for an inhomogeneus variant of the
  normalized $p$-{L}aplacian evolution.
\newblock {\em Nonlinear Anal.}, 121:458--468, 2015.

\bibitem[CIL92]{userguide}
M.~G. Crandall, H.~Ishii, and P.-L. Lions.
\newblock User's guide to viscosity solutions of second order partial
  differential equations.
\newblock {\em Bull. Amer. Math. Soc.}, 27(1):1--67, 1992.

\bibitem[DHHR11]{pxbook}
L.~Diening, P.~Harjulehto, P.~H{\"a}st{\"o}, and M.~R{\r u}{\v z}i{\v c}ka.
\newblock {\em Lebesgue and Sobolev Spaces with Variable Exponents}, volume
  2017 of {\em Lecture Notes in Mathematics}.
\newblock Springer-Verlag, 2011.

\bibitem[EG15]{measuretheoryevans}
L.~C. Evans and R.~F. Gariepy.
\newblock {\em Measure Theory and Fine Properties of Functions}.
\newblock CRC Press, revised edition, 2015.

\bibitem[IJS]{imbertJinSilvestre16}
C.~Imbert, T.~Jin, and L.~Silvestre.
\newblock H{\"o}lder gradient estimates for a class of singular or degenerate
  parabolic equations.
\newblock To appear in \textit{{A}dv. {N}onlinear {A}nal.}

\bibitem[Ish95]{Ishii95}
H.~Ishii.
\newblock On the equivalence of two notions of weak solutions, viscosity
  solutions and distribution solutions.
\newblock {\em Funkcialaj Ekvacioj}, 38:101--120, 1995.

\bibitem[JJ12]{newequivalence}
V.~Julin and P.~Juutinen.
\newblock A new proof for the equivalence of weak and viscosity solutions for
  the $p$-{L}aplace equation.
\newblock {\em Comm. Partial Differential Equations}, 37(5):934--946, 2012.

\bibitem[JL05]{quasiremove}
P.~Juutinen and P.~Lindqvist.
\newblock Removability of a level set for solutions of quasilinear equations.
\newblock {\em Comm. Partial Differential Equations}, 30:305--321, 2005.

\bibitem[JLM01]{equivalence_plaplace}
P.~Juutinen, P.~Lindqvist, and J.J. Manfredi.
\newblock On the equivalence of viscosity solutions and weak solutions for a
  quasi-linear equation.
\newblock {\em SIAM J. Math. Anal.}, 33(3):699--717, 2001.

\bibitem[JLP10]{equivalence_p(x)Laplace}
P.~Juutinen, T.~Lukkari, and M.~Parviainen.
\newblock Equivalence of viscosity and weak solutions for the
  $p(x)$-{L}aplacian.
\newblock {\em Ann. Inst. H. Poincar{\'e} Anal. Non Lin{\'e}aire},
  27(6):1471--1487, 2010.

\bibitem[JS17]{jinsilvestre17}
T.~Jin and L.~Silvestre.
\newblock H{\"o}lder gradient estimates for parabolic homogeneous
  $p$-{L}aplacian equations.
\newblock {\em J. Math. Pures Appl.}, 108(3):63--87, 2017.

\bibitem[Kat15a]{nikos}
N.~Katzourakis.
\newblock {\em An Introduction To Viscosity Solutions for Fully Nonlinear PDE
  with Applications to Calculus of Variations in $L^{\infty}$}.
\newblock Springer, 2015.

\bibitem[Kat15b]{convex_functionals}
N.~Katzourakis.
\newblock Nonsmooth convex functionals and feeble viscosity solutions of
  singular {E}uler-{L}agrange equations.
\newblock {\em Calc. Var.}, 54(1):275--298, 2015.

\bibitem[Koi12]{koike}
Shigeaki Koike.
\newblock {\em A Beginner's Guide to the Theory of Viscosity Solutions}.
\newblock 2nd edition, 2012.

\bibitem[Lin17]{lindqvist_plaplace}
P.~Lindqvist.
\newblock {\em Notes on the $p$-Laplace equation (second edition)}, volume 161
  of {\em Report / University of Jyv{\"a}skyl{\"a}, Department of Mathematics
  and Statistics}.
\newblock University of Jyv{\"a}skyl{\"a}, 2017.

\bibitem[MO]{chilepaper}
M.~Medina and P.~Ochoa.
\newblock On viscosity and weak solutions for non-homogeneous $p$-{L}aplace
  equations.
\newblock To appear in \textit{{A}dv. {N}onlinear {A}nal.}

\bibitem[PL13]{somepxpaper}
M.~P{\'e}rez-Llanos.
\newblock A homogenization process for the strong $p(x)$-{L}aplacian.
\newblock {\em Nonlinear Anal.}, 76:105--114, 2013.

\bibitem[PSSW09]{peresSchrammSheffieldWilson09}
Y.~Peres, O.~Schramm, S.~Sheffield, and D.B. Wilson.
\newblock Tug-of-war and the infinity {L}aplacian.
\newblock {\em J. Amer. Math. Soc.}, 22(1):167--210, 2009.

\bibitem[ZZ12]{strongpx_regularity}
C.~Zhang and S.~Zhou.
\newblock H{\"o}lder regularity for the gradients of solutions of the strong
  $p(x)$-{L}aplacian.
\newblock {\em J. Math. Anal. Appl.}, 389(2):1066--1077, 2012.

\end{thebibliography}

\Addresses
\end{document}